\documentclass{amsart}

\usepackage{amsmath, amsthm, amssymb, tikz, xypic}

\newtheorem{thm}{Theorem}[section]

\newtheorem{lem}[thm]{Lemma}
\newtheorem{cor}[thm]{Corollary}
\newtheorem{prop}[thm]{Proposition}

\theoremstyle{definition}
\newtheorem{dfn}[thm]{Definition}
\newtheorem{prb}[thm]{Problem}
\theoremstyle{remark}

\newcommand{\im}{\operatorname{im}}

\newcommand{\C}{\mathbb{C}}

\newcommand{\N}{\mathbb{N}}
\newcommand{\Z}{\mathbb{Z}}

\newcommand{\bb}[1]{\mathbb{{#1}}}
\newcommand{\inv}{^{-1}}
\newcommand{\ip}[1]{\langle {#1} \rangle} 
\newcommand{\sm}{\smallsetminus}
 
\newcommand{\tr}{\operatorname{tr}}

\newcommand{\dom}{\operatorname{dom}}

\newcommand{\aut}{\operatorname{Aut}}

\newcommand{\res}{\upharpoonright}

\newcommand{\supp}{\operatorname{supp}}
\newcommand{\cod}{\operatorname{cod}}
\newcommand{\var}{\operatorname{var}}
\newcommand{\cov}{\operatorname{cov}}

\newcommand{\lra}{\Leftrightarrow}

\newcommand{\s}[1]{\mathcal{{#1}}} 

\renewcommand{\deg}{\operatorname{deg}}

\renewcommand{\bf}{\mathbf}
\newcommand{\ax}{\mathsf{a}}
\newcommand{\bx}{\mathsf{b}}

\newcommand{\fix}{\operatorname{fix}}
\newcommand{\ob}{\operatorname{ob}}

\title{Limits of sparse hypergraphs}
\author{Riley Thornton}
\begin{document}

\begin{abstract}
    We generalize ultraproducts and local-global limits of graphs to hypergraphs and other structures. We show that the local statistics of an ultraproduct of a sequence of hypergraphs are the ultralimits of the local statistics  of the hypergraphs. Using some standard results from model theory, we conclude that the space of (equivalence classes of) pmp hypergraphs with the topology of local-global convergence is compact, and that any countable set of local statistics for a pmp hypergraph can be realized as the statistics of a set of labellings (rather than just approximated) in a local-global equivalent hypergraph.

    We give two applications. First, we characterize those structures where any solution to the corresponding CSP can be turned into a measurable solution. These turn out to be the width-1 structures. We can also use the limit machinery to extract from this theorem a purely finitary characterizations of width-1 structures involving asymptotic solutions.

    Second, we prove two measurable versions of the Frankl--R\"odl matching theorem using measurable nibble and differential equation arguments. The measurable proofs are much softer than the purely finitary results. And, we can recover the finitary theorems using the limit machinery.
\end{abstract}
\maketitle

\section{Introduction}

This paper extends the theory of local-global limits and ultraproducts to hypergraphs and more general structures. Unlike the theory of dense hypergraph limits, this turns out to be very straightforward. Later sections of the paper also give a handful of illustrative (and hopefully compelling) applications. Roughly, the limit theory gives a dictionary between measurable labellings of limit objects and statistics of labelled finite neighborhoods in the limiting objects. 

For example, an earlier theorem of the author characterizes trivial Borel CSPs. Underlying that result was a construction of acyclic Borel hypergraphs with large Borel chromatic number. The limit theory carries information about parameters like the probability a constraint is violated and the girth of hypergraphs. So, using the limit theory we can extend the theorem in two ways. First, using some recent results about independence numbers of random regular hypergraphs, we construct acyclic pmp hypergraphs with high measurable chromatic number. So, we get a characterization of trivial measurable CSPs. Second, we can translate this measurable theorem into a purely finitary theorem to the effect that for width-1 CSPs (and only for width-1 CSPs) there is a link between the size of the smallest unsolvable subinstance and the proportion of constraints a labelling has to violate. 

As another example, the limit theory can simplify arguments which involve the differential equation or nibble methods. In broad strokes, there are many arguments in probabilistic combinatorics which involve showing that a small random change in a labelling, or step, has a predictable effect in expectation. Then, some technical estimates allow one to find an actual trajectory, or sequence of steps, with the desired destination. In the limit, though, many of these technical estimates disappear. Regularity assumptions can be boosted, and the expected change is a literal change in measures. To illustrate this point, we re-prove the celebrated Frankl--R\"odl approximate matching theorem and give some numerical bounds on the density of a matching in the large girth case.

\subsection{Background}

We are generally interested in finding measurable solutions to combinatorial problems (about graphs, hypergraphs, equivalence relations, etc.)~where the vertex set forms a measure space. For instance, if I have a pmp action of $F_n$, what's the largest measurable independent set I can find in its Schreier graph? How many colors do I need to measurably color the Schreier graph? These kinds of problems arise naturally in ergodic theory. For graphs of bounded degree and problems where an arbitrarily small error is acceptable (like finding large independent sets), there is a well-established connection between measurable problems in the limit with asymptotic problems about sequences of graphs. 

Let's briefly recall this established connection for sparse graphs. There are a few different approaches. This paper will generalize the model theoretic and combinatorial approaches.

One can view a finite graph $G=(V,E)$ as a continuous model theoretic structure by instead looking at $\s P(V)$ as a Boolean algebra equipped with counting measure and the full group of involutions contained in $E$. The Boolean algebra with measure turns out to be interdefinable with $L^2(V)$ as a Hilbert space equipped also with pointwise multiplication and complex conjugation. So, the model theory of graphs viewed in this way is concerned with the statistics of labellings of the graph. For instance, the independence ratio and matching ratio will be continuous-first-order definable. And, ultralimits of graphs look like measure algebras (or equivalently tracial von Neumann algebras) equipped with involutions. We can find point realizations of the ultralimits, which means we can identify them with Schreier graphs of certain nonstandard measure preserving group actions. (There are some subtleties elided here, of course. For instance, there isn't a natural way to enumerate the involutions generating the edge set.)

In combinatorics, we can put a metric on graphs by declaring two graphs to be close if the empirical distribution of isomorphism types of colored radius $r$ neighborhoods in any labelling of one graph can be well-approximated by a labelling of the other. So, if two graphs are close, their independence ratios, matching ratios, and so on will all be close. (This is the colored neighborhood metric of Bollobas and Riordan). Hatami, Lovasz, and Szegedy noticed that this definition can also be applied to graphs on measure spaces (though one only ends up with a pseudometric), and that the space of pmp graphs of a bounded degree is complete (indeed, compact) with respect to this metric. Limits in this metric are called local-global limits.

These two approaches are really the same: two pmp graphs are at distance 0 if and only their continuous models are elementary equivalent, and any ultralimit of a sequence of pmp graphs of bounded degree will be equivalent to a local-global limit of some subsequence. This is spelled out in \cite{LocalGlobalBackground}. In fact this gives another proof of the Hatami--Lovasz--Szegedy theorem. And, in either setting, one can use standard compactness arguments to transfer results between sequences of finite graphs and pmp graphs. For instance, one can show that any $d$-regular acyclic pmp graph has an independent set of density $\alpha_d=\frac{\log(d)}{d}\bigl(1+o(d)\bigr)$. It then follows by compactness that, for $\epsilon>0$, any $d$, and any large enough $g$, any $d$-regular graph with girth $g$ has an independent set of density $\alpha_d-\epsilon$.

\subsection{Statement of results}

In the first section of this paper, we will adapt the theory above to the setting of hypergraphs. Unlike the theory of dense graph limits, this generalization turns out to be very straightforward. For the most part it consists of changing some $2$'s to $u$'s. 

A pmp hypergraph is a uniform hypergraph on a standard probability space whose connectedness relation is pmp. The (spectrum of) continuous model(s) of a pmp hypergraph is defined analogously as for graphs: pick some pmp group action which generates edge set and consider $L_2(V,\mu)$ equipped with the associated automorphisms. As with graphs, if we fix a degree bound, then we can always use the same group (in this case a free power $(\Z/u\Z)$, where $u$ is the uniformity of the hypergraph). The local statistics of a pmp hypergraph are also defined exactly analogously to those of a pmp graph, as are the attendant notions of local-global limits and local-global containment. And, we have completely analogous theorems to those in the graph setting:

\begin{thm}
    For a sequence of pmp $u$-uniform hypergraphs $\ip{\s H_i: i\in\N}$ all of degree at most $\Delta$, and pmp hypergraphs $\s H$ and $\s G$, the following are equivalent:
    \begin{itemize}
        \item $\ip{\s H_i:i\in\N}$ local-global converges to $\s H$ 
        \item There is a sequence of markings $\ip{\mathsf{a}_i: i\in\N}$ which elementary converges to some $\mathsf{a}$ whose Schreier hypergraph is local-global equivalent to $\s H$.
    \end{itemize} 

    As are the following containment notions:
    \begin{itemize}
        \item $\s H$ local-global contains $\s G$
        \item There is $\s H'$ local-global equivalent to $\s H$, so that every marking of $\s G$ is weakly contained in a marking of $\s H'$
        \item There is $\s H'$ local-global equivalent to $\s H$ and some marking of $\s G$ is contained in a marking of $\s H'$
    \end{itemize}
    
    \end{thm}

And, as a corollary, we have closure and continuity properties for the space of (equivalence classes of) pmp graphs with the topology of local-global convergence.

\begin{thm}
    If each $\s H_i$ is a $u$-uniform hypergraph of degree at most $d$, and $\ip{\s H_i:i\in\N}$ is Cauchy sequence in the local-global metric, $\ip{\s H_i:i\in\N}$ local-global converges to some pmp hypergraph.
\end{thm}
The topology of local-global convergence is inherited from a compact metrizable topology, so this version of completeness implies compactness.

\begin{thm}
    If $\phi$ is some hypergraph parameter that can be detected by a formula in $L^2(V,\mu)$, then $\phi$ is continuous with respect to local-global convergence.
\end{thm}
So, for instance, the spectral radius of a hypergraph is continuous.

Important in the theory of local-global containment is the existence of a minimal element (with a given local geometry). This goes through for hypergraphs as well

\begin{thm}
    If $\s H$ is an atomless pmp hypergraph of bounded degree with local geometry $\mu_0$, then $\s H$ local-global contains the Bernoulli hypergraphing $\s B(\mu_0)$. 
\end{thm}

In the next two sections, we give two applications of this machinery. The first is to constraint satisfaction problems. For a finite relational structure $\s D$, the associated constraint satisfaction problem, or CSP, is the following: given a structure $\s X$ with the same signature (called an instance of $\s D$), determine whether or not there is a homomorphism from $\s X$ to $\s D$ (called a solution to $\s X$). The class of CSPs (as $\s D$ varies) includes many classical problems: $k$SAT, graph $n$-coloring, systems of linear equations over a finite field, etc.  Building on previous work of the author, we can characterize which measurable CSPs are no more complicated than their classical counterparts.

\begin{thm}
    For a finite relational structure $\s D$, the following are equivalent:
    \begin{enumerate}
        \item $\s D$ is width-1
        \item Any pmp instance of $\s D$ with a solution has a measurable solution
        \item For any bounded degree sequence of instances of $\s D$, $\ip{\s X_i: i\in\N}$, either we can eventually satisfy all but an $\epsilon$ fraction of the constraints on $\s X_i$, or for some $n$ there are arbitrarily large $i$ so that $\s X_i$ has a sub-instance of size at most $n$ with no solution.
    \end{enumerate}
\end{thm}

And lastly, we prove a measure theoretic generalization of celebrated Frankl--R\"odl matching theorem:

\begin{thm}
    For any $u\geq 2$, $\epsilon>0$, large enough $\Delta$, small enough $\delta$, and any $\lambda$, if $\s H$ satisfies
    \begin{enumerate}
        \item $\bb P(|\deg(v)-\Delta|<\delta\Delta)<\delta$ (almost regularity)
        \item $\bb P((\exists u)\;\cod(u,v)>\delta \Delta)<\delta$ (small codegree)
        \item $\bb P(\deg(v)>\lambda)=0$ (bounded degree)
    \end{enumerate}
    Then $\s H$ has a matching that covers a $(1-\epsilon)$-fraction of vertices.

    Further, if $\s H$ is $\Delta$-regular and acyclic, then $\s H$ has a matching covering a set of vertices of measure
    \[1-\left(\frac{1}{(u-1)d}\right)^{u-1-\frac{1}{d}}.\]
\end{thm}

The point of this theorem isn't that it's a particularly impressive generalization. Rather, the proofs here are quite a bit softer than the purely finitary proofs of the corresponding asymptotic finite statements. And, we can recover the Frankl--R\"odl theorem via a compactness argument.

\subsection{Notation}

This paper is largely about hypergraphs. All of our hypergraphs will be simple and uniform. So, by a hypergraph we mean $H=(X, E)$ where $E\subseteq \binom{X}{u}$ for some $u$. We call $u$ the uniformity of $H$. The elements of $E$ are called edges or, if we want be emphatic, hyperedges. Sometimes we will abuse notation and write $H$ for the set of edges in $H$.

We will also want to refer to ordered edges, we write $\vec E$ or $\vec H$ for the set $\{(x_1,...,x_u): \{x_1,...,x_u\}\in E\}$. We also call these edges. More generally, for a relation $R$, we will call a tuple $(x_1,...,x_n)$ so that $R(x_1,...,x_n)$ an edge in $R$.

If $\Gamma=\ip{\gamma_1,...,\gamma_n}$ is a marked group with generating set $E=\{\gamma_1,...,\gamma_n\}$ so that, for each $i$, $\gamma_i$ has order $u$, and $\ax: \Gamma\curvearrowright X$ is an action, then we can define an associated hypergraph with uniformity $u$ and degree at most $n$ as follows:
\[\s H(\ax):=(X, \{\{\gamma_i^j\cdot x: j\in [u]\}: i\in [n], x\in X, (\forall j<u)\;  x\not\in \fix(\gamma_i^j)\}).\]

By a cycle in a hypergraph, we mean a Berge cycle, or a tuple of distinct vertices $(v_1,...,v_n)$ and distinct edges $(e_1,...,e_n)$ so that, for all $i$, $v_i, v_{i+1}\in e_i$ with indices taken mod $n$. For two vertices $u,v$ of a hypergraph, their codegree in $H$ is
\[\cod(u,v)=|\{e\in H: u,v\in e\}|.\] So, if $\cod(u,v)>1$, there is a cycle of length $2$ in $H$. The girth of a hypergraph is the smallest length of a cycle (or $\infty$ if the hypergraph is acyclic).

We will talk about isomorphism classes of rooted labelled hypergraphs. By a rooted labelled hypergraph, we mean $(H,o,g)$ where $H$ is a hypergraph, $o$ is some distinguished vertex, and $g$ is function from some subset of the vertices and edges in $H$ to a set of labels. We say $(H,o,g)\cong (G,v,f)$ when there is an isomorphism $\phi$ from $H$ to $G$ that sends $o$ to $v$ and so that $g=f\circ \phi$. We write $[H,o,g]$ for the isomorphism class of $(H,o,g)$.

For a set $A$, its we write $\bf 1_A$ for its characteristic function
\[\bf 1_A(x)=\left\{\begin{array}{cc}1 & x\in A\\
0 & else\end{array}\right.\] (The domain of $\bf 1_A$ should always be clear from context.

Aside from characteristic functions, we reserve bold letters for random variables. We write $\bb P(\s A)$ for the probability of an event $\s A$. If we need to specify the random variable or measure we are taking this probability with respect to, we add a subscript. For instance, if $\mu$ is probability measure on vertices of a hypergraph, and $f$ is a measurable labelling of the hypergraph, we might write
\[\bb P_\mu(f(v)=2)\] for \[\mu(\{v: f(v)=2\}).\]

\subsection{Acknowledgements} Thanks to Clinton Conley and Felix Weilacher for helpful conversations and comments. This work was partially supported by NSF MSPRF grant DMS-2202827.

\section{General nonsense}
We'll start by lifting the general theory of local-global convergence and ultralimits from the setting of pmp graphs to the setting of pmp hypergraphs. Everything goes through (with some notational inconvencience) for pmp structures.

\subsection{Definitions: pmp relations and local-global convergence}

Probability measure preserving (henceforth pmp) graphs admit a handful of equivalent definitions. We can extend these to arbitrary relational structures. Underlying all of these definitions is the notion of a pmp equivalence relation.
\begin{dfn}
    A \textbf{pmp equivalence relation} is a equivalence relation $E$ on a standard probability space $(X,\mu)$ so that each equivalence class is countable and whenever $f: X\rightharpoonup X$ is a Borel partial injection that sends points to $E$-related points and $A\subseteq \dom(f)$ is measurable, then $\mu(f[A])=\mu(A)$.
\end{dfn}

If $E$ is pmp, then certainly any list of injections generating $E$ will be a list of measure preserving maps (and the Lusin--Novikov theorem tells we can always find generating maps). Conversely, if we have a family of generating maps which are measure preserving then any injection can be chopped up into piece-wise applications of these generators. So, we get the following:

\begin{prop}[{\cite[Chapter 1]{KM}}]
    $E$ is a pmp equivalence relation on $(X,\mu)$ if and only if $E=\bigcup_{i\in \N} f_n$, where each $f_n$ is a $\mu$-perserving involution.
\end{prop}

If we have a pmp equivalence relation $E$, then we can also measure relations that live inside $E$.

\begin{dfn}
    For an equivalence relation $E$ on $X$, \[[E]^n=\{(x_1,...,x_n)\in X^n: x_1 \, E\, x_2\, E\,...\, E\,x_n\}.\]

Suppose $\mu$ is an invariant measure for $E$. For each $n\in \N$, we define a measure on $[E]^n$ as follows: for $A\subseteq [E]^n.$
    \[\tilde\mu(A):=\int_{x\in X} \bigl|\pi_1\inv[x]\cap A\bigr|\;d\mu\]
\end{dfn}

We have a measure-theoretic handshake lemma (that actually characterizes pmp equivalence relations). Integrating the size of the section of a relation over any coordinate should give the same answer.

\begin{prop}
    For any $i\in [n]$ and $A\subseteq [E]^n$,
    \[\tilde \mu(A)=\int_x \bigl|\pi_i\inv[x]\cap A\bigr|\;d\mu\] And if $\tilde \mu$ is invariant in this way, then $E$ must preserve $\mu$.
\end{prop}
\begin{proof}
    By Luzin--Novikov and some standard coloring arguments, we can write $A$ as $A=\bigcup_{k\in \N} f_k$ where $f_k: X\rightarrow X^{n-1}$ and $\pi_i\circ f_k$ is injective for all $k$. Then, for any $k$,
    \[\mu\bigl(\dom(f_k)\bigr)=\mu\bigl(\im(\pi_i\circ f_k)\bigr)\] so
    \begin{align*}
        \int_{x\in X}\bigl|A\cap \pi_1\inv[x]\bigr|\;d\mu & = \sum_{k\in\N} \mu\bigl(\dom(f_k)\bigr) \\
        & = \sum_{k\in\N} \mu\bigl(\im(\pi_i\circ f_k)\bigr) \\
        & = \int_{x\in X} \bigl|A\cap \pi_i\inv[x]\bigr|\;d\mu.
    \end{align*} The converse is true already if $\tilde \mu$ is invariant for relations of arity 2. Details are in \cite{KM}.
\end{proof}

\begin{dfn}
    A \textbf{pmp relation} is a Borel subset $R\subseteq[E]^n$ for some pmp equivalence relation $E$ and some $n\in \N$. A \textbf{pmp structure} is a standard probability space equipped with a list of pmp relations.
\end{dfn}

So, for example, a pmp directed graph is a measurable orientation of a pmp graph. Since pmp equivalence relations are closed downwards, we can always take $E$ to be the equivalence relation generated by 2D projections of $R$.

Building our definitions on top of pmp equivalence relations constrains what sort of structures we can end up with. For instance, there is no pmp directed acyclic graph where each vertex has outdegree 1 and in degree 2: integrating the edge set over the first coordinate would give us $1$ while integrating over the second coordinate would give us $2$. These restrictions are appropriate to settings where we want to lift finite combinatorics into the measurable setting, since these constraints mirror those constraints that counting arguments place on finite structures.

In the next section, we will borrow tools from ergodic theory to develop a limit theory for pmp structures. Ergodic theory has a good notion of limit for measure preserving functions $f: X\rightarrow X$ and measurable maps $f: X\rightarrow \C$ (namely ultraproducts). If $R$ is a pmp relation, we can enumerate measure-preserving functions $f_{i,j}$ so that $R=\{(x, f_{i1}(x),...,f_{in-1}(x)): x\in X\}$ and code edge-labellings of $R$ as vertex labellings using these addresses for its edges. So, we can apply the ergodic theory machinery. To ease notation, we'll specialize to the case of hypergraphs.

\begin{dfn}
    A \textbf{pmp hypergraph} on $(X,\mu)$ is a hypergraph on $(X,\mu)$ that is a pmp relation.
\end{dfn}

Some key examples:

\begin{prop}
    Any finite hypergraph is pmp, and normalized counting measure is invariant.
\end{prop}
\begin{proof}
    Counting measure is preserved by all injections.
\end{proof}

\begin{prop}
    If $\ax:(\Z/u\Z)^{*n}\curvearrowright X$ is a pmp action, then the induced hypergraph 
    \[\s H(\ax)= (X, \{\{x, \gamma\cdot x, ..., \gamma^{u-1}\cdot x\}: x\in X, \gamma\in E\})\] is a pmp hypergraph with uniformity $u$. (Recall $E$ is the standard generating set for $(\Z/u\Z)^{*n}$)
\end{prop}
\begin{proof}
Given any Borel injection $f$ so that $f(x)\in \Gamma\cdot x$, we can set
\[U_{\gamma}=\{x: \gamma\cdot x=f(x)\}.\] Then, for $\gamma, \delta$ distinct, since $f$ is an injection, $f(U_\gamma)\cap f(U_\delta)=\emptyset$. For any $A\subseteq \dom(f)$, 
\begin{align*}\mu(f[A])=\mu(\bigcup_{\gamma} f[U_\gamma\cap A])= &\sum_{\gamma\in \Gamma}\mu(f(U_{\gamma\cap A})) \\
&= \sum_{\gamma\in \Gamma} \mu(\gamma\cdot U_{\gamma}\cap A)\\
&=\sum_{\gamma\in \Gamma} \mu(U_\gamma\cap A) = \mu(\bigcup_\gamma U_\gamma\cap A)=\mu(A). \end{align*}
\end{proof}

These examples suggest links with dynamics and combinatorics. Moreover, every pmp hypergraph secretly of the kind indicated in the second proposition above.

In \cite{LocalGlobalBackground}, the correspondence between actions of groups and marked graphs gives a uniform approach to various convergence notions in dynamics, combinatorics, and logic. A marking of a graph is a specific kind of enumeration of functions generating the graph. The main point are that, for graphs of uniformly bounded degree, we can find a single group which generates the graphs. In particular, we have a uniform way of traversing neighborhoods in these graphs. We can find analogous markings for hypergraphs of bounded degree and fixed uniformity.

\begin{dfn}
    A \textbf{marked hypergraph} is a pmp hypergraph $\s H$ along with a group action $\ax:\Gamma\curvearrowright X$ so that $\s H=\s H(\ax)$. We say $\ax$ is a marking with $\Gamma$.
\end{dfn}

Markings of a hypergraph may be wildly different from the point of view of actions. And there are many combinatorial questions one could raise about the possible markings of a given hypergraph. But as we are just using markings for bookkeeping purposes, the choice of marking will be largely irrelevant in this paper.

\begin{thm}
    For every $d,u\in\N$, every pmp hypergraph $\s H$ with degree at most $d$ and uniformity $u$ has a marking with $\Gamma=(\Z/u\Z)^{*(ud-u+1)}$
\end{thm}
\begin{proof}
 Consider any $u$-uniform pmp hypergraph $\s H=(X,E)$ so that every vertex is only contained in at most $d$ edges. By a greedy coloring bound on the intersection graph, we can measurably color the edges in $\s H$ using $ud-u+1$ colors so that any edges with the same color are disjoint \cite[Proposition 4.6]{KST}. Pick a Borel linear order, $<$, on $X$. Then, let $\Gamma=\ip{a_i: i\in\N, a_i^u=e}$ act on $(X,\mu)$ by $a_i\cdot x=y$ if there is $e=\{x_1,...,x_u\}\in H$ with $x_1<x_2<...<x_u$ so that $e$ gets color $i$, and for some $j$, $x=x_j$ and $y=x_{j+1}$ (mod $u$).
\end{proof}

We want to study the convergence of local statistics of measurable labellings of pmp hypergraphs.

\begin{dfn}
    Fix $r,d,u\in\N$ and compact space $X$. Let $\bb L(r, X, d,u)$ be the set of isomorphism classes of finite rooted $u$-uniform hypergraphs of degree at most $d$ with vertex and edge $X$-labellings (that's isomorphism as rooted labelled hypergraphs). For $k\in \N$, write $\bb L(r,k,d,u)$ for $\bb L(r,[k],d,u)$. We denote elements of $\bb L(r,k,d,u)$ by $[(H,o,g)]$ where $H$ is a hypergraph, $o$ is a root, and $g$ is labelling.  Unless they are relevant, we will suppress $d,u$ in the notation and write $\bb L(r,k).$
\end{dfn} Note that $\bb L(r,k)$ is finite. So, the set of probability measures on $\bb L(r,k)$, $P_1\bigl(\bb L(r,k)\bigr)$, is compact with respect to any reasonable metric, and any two reasonable metrics will be equivalent. We'll stick with the total variation metric for this paper:
\[d_{TV}(\mu,\nu)=\sum_{[H,o,g]\in\bb L(r,k)} \bigl|\mu\bigl([H,o,g]\bigr)-\nu\bigl([H,o,g]\bigr)\bigr|.\]

\begin{dfn}

    For a pmp hypergraph $\s H=(X, E)$ on a probability space $(X,\mu)$, and $f:X\cup  \vec E\rightarrow [k]$ measurable\footnote{For edges, we mean measurable with respect to $\tilde \mu$ defined earlier.}, the \textbf{$r$-local statistics} of $f$ is the measure on $\bb L(r,k)$ given by
    \[\mu^r(f, \s H)\bigl([H,o,g]\bigr)=\mu\bigl(\{x: (H,o,g)\cong(B^{\s G}_r(x), x, f)\}\bigr).\] When $\s H$ is clear from context, we will suppress it.

    The set of \textbf{exact $(r,k)$-local statistics} of $\s H$ is set of all $r$-local statistics of $k$-labellings of $\s H$:
    \[\s L'_{r,k}(\s H)=\{\mu^r(f): f\in[k]^{X\cup \vec E}\mbox{ measurable}\}\] And, the set of \textbf{$(r,k)$-local statistics} is closure of the set of exact $(r,k)$-local statistics:
    \[\s L_{r,k}(\s H)=\overline {\s L'_{r,k}(\s H)}\]

\end{dfn}

The $(r,k)$-local statistics of a hypergraph detect many dynamical and combinatorial parameters. For instance: 
\begin{dfn} A set of edges in $\s H=(X, H)$ is \textbf{independent} if it does not contain an edge. That is, $A$ is independent if $(\forall x_1,...,x_u\in A)\;\{x_1,...,x_u\}\not\in H$. Or, equivalently, $A^n\cap \vec H=\emptyset$. The \textbf{independence ratio} of $\s H$ is the supremum of the measures of measurable independent sets:
\[\alpha_\mu(\s H):= \sup\bigl\{\mu(A): A\subseteq X, \;A\mbox{ is independent}\bigr\}.\] 
\end{dfn} This is detected by $\s L_{1,2}(\s H)$. First, let $\s I$ be the set of measures in $P_1\bigl(\bb L(r,k)\bigr)$ supported on independent sets: 
\[\s I=\biggl\{\nu\in P_1\bigl(\bb L(1,2)\bigr): \nu\bigl(\{[H,o,f]:f\inv[1]\mbox{ is independent in }H \}\bigr)=1\biggr\}\] Then, for any $\s H$
\[\alpha_\mu(\s H)=\sup\bigl\{\nu(\{[H,o,g]: g(0)=1\}): \nu\in \s L_{r,k}(\s H)\cap \s I\bigr\}. \] (Hiding here is the observation that if $\nu\in \s I\cap \s L_{rk}(\s H)$, then $\nu$ is approximated by measures in $\s L_{rk}'(\s H)\cap \s I$.) 

More generally, the maximum density of a set on which a measurable labelling of a hypergraph can satisfy a locally checkable constraint (in the sense of \cite{gridsSurvey}) is detected by $\s L_{r,k}$, where $k$ is the label set and $r$ is the radius needed to check the local constraint. 

Associated to these statistics are notions of containment, equivalence and convergence

\begin{dfn}
    For two pmp hypergraphs $\s H$ and $\s G$, $\s H$ \textbf{local-global contains} $\s G$ if for each $r,k$ the set of $(r,k)$-local statistics of $\s H$ contain those of $\s G$. In symbols:
    \[\s G\preccurlyeq_{lg} \s H:\lra (\forall r,k)\; \s L_{rk}(\s G)\subseteq \s L_{rk}(\s H).\]
\end{dfn}

So, $\s H$ local-global contains $\s G$ if any labelling of $\s G$ can be locally simulated to arbitrary precision by a labelling of $\s H$. Thus, if $\s G\preccurlyeq_{lg} \s H$, then $\alpha_\mu(\s G)\leq \alpha_\mu(\s H)$.

\begin{dfn}
    We say two hypergraphs are \textbf{local-global equivalent} if they have the same local statistics. That is
    \[\s H\equiv_{lg} \s G:\lra \s H\preccurlyeq_{lg} \s G\mbox{ and } \s G\preccurlyeq_{lg} \s H.\]
    A sequence of hypergraphs $\ip{\s H_i:i\in\N}$ is \textbf{local-global Cauchy} if all of their degrees are bounded (say by $d$) and, for every $r,k\in \N$, the sequence $\ip{\s L_{rk}(\s H_i):i\in\N}$ converges in the Hausdorff metric on closed subsets of $P_1(\bb L(r,k)).$ And the sequence \textbf{local-global converges} to $\s H$ if $\s L_{rk}(\s H_i)$ converges to $\s L_{rk}(\s H)$ for all $r,k$.
\end{dfn}

For example, if $\s H$ and $\s G$ are hyperfinite and all of their components are isomorphic, then they are local-global equivalent: you can simulate any labelling you like on large component finite pieces with small boundary. If $\s H$ is an inverse limit of finite hypergraphs (equipped with the limit of counting measures), then $\s H$ is also a local global limit: any labelling of $\s H$ is within $\epsilon$ of a continuous labelling, which will factor through some finite hypergraph in the limiting sequence. For any $\s H$, the graph $\s H\sqcup \s H$ local-global contains $\s H$, and if $\s H$ is ergodic, this containment is strict.

In many applications, we're interested in local-global limits of sequences of finite hypergraphs. If the size of the finite hypergraphs goes to infinity, we end up with an infinite (indeed atomless) pmp hypergraph. Often in the limit various irregularities of the finite sequence are smoothed out. For instance, any limit of random regular hypergraphs will almost surely be acyclic.

In the next subsection, we'll prove compactness, representation, and continuity theorems for local-global convergence.

\subsection{Marked hypergraphs and borrowing from ergodic theory} 

Folklore has it that local-global limits of graphs, weak containment of actions, and $\Sigma_1$-equivalence in continuous model theory are essentially different languages for capturing the same information, where we view an action as a graph with special kind of edge labelling. See \cite{LocalGlobalBackground} for a careful development of these ideas. As indicated in the previous subsection, there's a similar correspondence between edge-labelled hypergraphs and actions of groups generated by torsion elements. All of the theory goes through with a few obvious and benign modifications.  We'll state all the main points here, and give sketches of the main proofs.

Fix a degree bound $d$, a uniformity $u$, and a marked group $\Gamma=\ip{\gamma_1,...,\gamma_{ud-d+1}}$ so that each $\gamma_i$ is of order $u$ and any $u$-uniform hypergraph of degree at most $d$ has a marking with $\Gamma$. We will work with this data throughout the section.

Viewing actions as labelled (hyper)graphs, we get a definition of local statistics for actions as well. 

\begin{dfn}
Suppose $E=\ip{\gamma_1,...,\gamma_n}$ (so, $\Gamma=\ip{E}$ and each element of $E$ is order $u$). For an action $\ax:\Gamma\curvearrowright X$ with Schreier hypergraph $\s H$, the associated \textbf{diagram} is \[d_{\ax}:\s H\rightarrow\s P( E)\]
\[d_{\ax} (x_1,...,x_u)=\{(\gamma, i_2,...,i_u):(\forall n)\; \gamma^{i_n}\cdot x_1=x_n\}.\] That is, $d_{\ax}$ labels an edge with the set of generators which traverse that edge, and the order in which they traverse it.

The $r$-\textbf{local statistics} of a vertex labelling of $f:X\rightarrow [k]$ with respect to $\ax$ is the measure $\mu^r(f, \ax):=\mu^r((f,d_{\ax}), \s H)\in P_1\bigl(\bb L(r,k\times \s P(E))\bigr)$. (Here, $(f, d_{\ax})$ is the obvious joint labelling). The notions of $(r,k)$-local statistics and local-global convergence, containment, and equivalence for actions are all defined analogously to the same for hypergraphs, but using this definition of local statistics for (vertex) labellings.
\end{dfn} 

The local statistics for an action only consider vertex colorings (and the diagram of the action). This is no weaker than tracking edge-labellings too (or even higher-arity labellings), since once we have a marking for a hypergraph, we have canonical addresses for all of its edges, so the $(r,k)$-local statistics of an action are already captured by it's vertex labellings. And, it means we can study the statistics of actions on $(X,\mu)$ by looking at $L^2(X,\mu)$.

\begin{prop}\label{prop:coding}
    For any measure $\nu\in \bb L(r,k)$, there is some $k'\in\N$ and a measure $c(\nu)\in \bb L(r,k')$ so that, for any action $\ax$, $\nu\in \s L_{rk} (\s H(\ax))$ iff $\nu'\in \s L_{rk'}(\ax).$
\end{prop}

\begin{proof}
    For oriented edge $e=(v_1,...,v_u)$ in $\s H(\ax)$, there is a generator and an list of powers, $(\gamma, \sigma(2),...,\sigma(u))$ so that $e=(v_1, \gamma^{\sigma(2)}\cdot v_1,..., \gamma^{\sigma(u-1)}\cdot v_1).$ We can code an edge labelling by recording at each vertex $v$, the edge-label associated to each $(\gamma, \sigma).$  So, we can set $c(\nu)\bigl([H,o,(g,d)]\bigr)$ to be the probability that $d$ is the local picture of the diagram for $\ax$ and that $g$ is the local picture of the vertex coding of a labelling sampled from $\nu$.
\end{proof}
\begin{cor}
    If $\ax$ local-global contains $\bx$, then $\s H(\ax)$ local-global contains $\s H(\bx).$
\end{cor}

One advantage of actions is that they come with a natural notion of ultraproduct. Or, if it's not entirely natural, a notion which fits into a well-studied framework from functional analysis. First we define the ultraproduct of (actions on) Hilbert spaces, then we define the ultraproduct of the actions to be a point realization of the Hilbert space ultraproduct. (One could also use the action on the measure algebra; this turns out to be equivalent).

\begin{dfn} When we write $L^2(X,\mu)$, we mean the Hilbert space equipped with the usual metric $d(f,g)^2=\sqrt{\ip{(f-g),(f-g)}}$, pointwise multiplication and conjugation, a unit $\bf 1:=\bf 1_X$, and inner products. (If you like, this is a commutative tracial {von Neumann} algebra.)

Enumerate $\Gamma$ as $\Gamma=\{\delta_1,\delta_2,...\}$. Suppose that $E=\{\gamma_1,...,\gamma_n\}$ is our distinguished generating set. Given a (possibly nonstandard) measure preserving action $\ax:\Gamma\curvearrowright (X,\mu)$, define the associated \textbf{continuous model}, $M(\ax)$, by 
\[M(\ax)=(L^2(X,\mu), \ax(\gamma_1),...,\ax(\gamma_n),\bf 1_{\fix(\delta_1)}, \bf 1_{\fix(\delta_2)},...)\]
    where we abuse notation slightly and write $\gamma$ for the autormorphism of $L^2(X,\mu)$ given by the action: $\gamma(f)=\gamma\cdot_{\ax} f,$ so $\bigl(\gamma(f)\bigr)(x)=f(\gamma\inv\cdot_{\ax} x)$. We write $\bf 1$ for the unit $\bf 1= \bf 1_X=\bf 1_{\fix(e)}$ and give $M(\ax)$ the trace defined by $\tr(x)=\ip{x, \bf 1}.$

For a sequence of actions $\ax_i:\Gamma\curvearrowright (X_i,\mu_i)$ and ultrafilter $\s U$ on $\N$, the \textbf{ultraproduct} of the $M(\ax_i)$s is:
    \[\lim_{\s U} M(\ax_i)=\bigl (\lim_{\s U} L^2(X_i,\mu_i), \lim_{\s U} \ax_i(\gamma_1),...,\lim_{\s U} \ax_i(\gamma_n), \lim_{\s U} \bf 1_{\fix(\delta_1)},...\bigr).\] Here  $\lim_{\s U} L^2(X_i,\mu_i)$ is the ultraproduct as a tracial von Neumann algebra, i.e. the structure with domain \[\{\ip{x_i:i\in\N}: (\exists N)(\forall i)\; x_i\in M(\ax_i)\mbox{ and }|x_i|<N\}/\sim_{\s U}\] where \[ x \sim_{\s U} y:\lra \lim_{\s U} d(x_i,y_i)=0.\]
    
    We call the $\sim_{\s U}$ equivalence class $[x]_{\s U}$ the ultralimit of the sequence $x$.
    
    And all operations are defined as pointwise ultralimits. For instance:
    \[[x]_{\s U}+[y]_{\s U}= [\ip{x_i+y_i:i\in\N}]_{\s U}\] 
    \[ \ip{[x]_{\s U}, [y]_{\s U}}=\lim_{\s U}\ip{x_i, y_i}\]and \[\ax^{\s U}(\gamma)\bigl([x]_{\s U}\bigr)=[\ip{\ax_i(\gamma)(x_i):i\in\N}].\]

    We write $M(\ax)^{\s U}$ for the case where all the $\ax_i$ are identical to $\ax$. In this case, we call $M(\ax)^{\s U}$ the \textbf{ultrapower} of $M(\ax)$.
\end{dfn}

We can recover $(X,\mu)$ from $M(\ax)$ as follows: measurable sets in $X$ correspond to functions in $M(\ax)$ satisfying $f^2=f$, intersections correspond to products, complements correspond to subtraction from $\bf 1$, the measure corresponds to the inner product against $\bf 1$. More generally, given any commutative tracial {von Neumann} algebra, we can build a Boolean algebra with a premeasure in this way. If we have a separable algebra $\s A$ equipped with countably many automorphisms, we can find a point realization, i.e.~a pmp group action $\ax$ so that $M(\ax)\cong \s A$.

One can check that $\lim_{\s U} M(\ax_i)$ is a commutative tracial {von Neumann} algebra equipped with a set of automorphisms which generate an action of $\Gamma$ (the algebraic identities defining such algebras only involve operations that commute with ultralimits). It turns out that $\lim_{\s U} M(\ax_i)$ is isomorphic to $M(\ax)$ for some action of $\ax:\Gamma\curvearrowright (\Omega,\mu)$ on a nonstandard probability space. We'll take this point realization $\ax$ as our definition of the ultraproduct of actions.

\begin{dfn}
    We say a nonstandard action $\ax$ is an ultraproduct of a sequence of pmp actions $\ip{\ax_i:i\in\N}$ if $M(\ax)$ is isomorphic to $\lim_{i\in\s U} M(\ax_i)$. In this case, we write $\ax=\lim_{i\in\s U} \ax_i.$ If $\ax_i=\bx$ for all $i$, we write $\ax=\bx^{\s U}$.
\end{dfn}

Some basic theorems of continuous model theory tell us how these ultraproducts behave and how to extract standard actions from them. Explaining the model theory in detail would mean developing a rather involved language. This is done in \cite{LocalGlobalBackground}, but here I will just state some consequences for local-global convergence. For readers who are already indoctrinated, I will say that in what follows $\s L_{rk}(\ax)$ is equivalent to the $\Sigma_1$-theory of $\ax$, and one can show that $\Sigma_1$-equivalence for actions implies elementary equivalence.

\begin{thm}\label{prop:fakelos}
    For all $r,k$, $\s L_{rk}(\lim_{\s U}\ax_i)=\lim_{\s U} \s L_{rk}(\ax_i)$, where the latter limit is taken in the Hausdorff metric.
\end{thm}

\begin{proof}
    We can characterize ultralimits in the Hausdorff metric as follows: for $\ip{E_i:i\in\N}$ a sequence in $K(X)$ and $x\in X$,
    \[x\in \lim_{\s U} E_i \lra (\exists x_i\in E_i)\;\lim_{\s U} d(x, x_i)=0.\]
    For an action $\ax:\Gamma\curvearrowright X$, and $f:X\rightarrow [k]$, let $e_i$ be the characteristic function of $f\inv[i]$. So $f=\sum_{i\in[k]} i e_i$. Fix a finite labelled rooted hypergraph $(H,o,g)$, let $Q$ be the set of labellings of $B_r(e)$ and equivalence relations so that the quotient is isomorphic to $(H,o,g)$:
    \[Q=\{(h,\sim): (B_r(e)/\sim, [e]_{\sim}, h/\sim)\cong (H,o,g)\}. \] So,
   \begin{align*} 
        \mu^r(f)\bigl( \{[H,o,g]\} \bigr) & = \sum_{(\sim, h)\in Q} \bb P \bigl( (\forall \gamma,\delta\in B_r(e))\;(\gamma\sim\delta \lra \gamma\cdot x=\delta\cdot x)\wedge f(\gamma)=h(\gamma) \bigr)\\
        & = \sum_{(\sim, h)\in Q} \tr\left( \prod_{\gamma\sim \delta} \bf 1_{\fix(\gamma\delta\inv)} \prod_{\gamma} \gamma\inv\cdot e_{h(\gamma)} \right).
    \end{align*} So the total variation distance $d_{TV}(\mu^r(f), \nu)$ is a sum of traces of simple algebraic functions of $e_1,...,e_n$.

    Traces, products, etc.~all commute with ultralimits, so for any sequence $\ip{f_i:i\in\N}$, where $f_i$ is a labelling of an action $\ax_i$, \[d_{TV}(\mu^r([\lim_{i\in \s U} f_i]), \nu)=\lim_{i\in\s U} d_{TV}(\mu^r(f_i),\nu).\] This means the right hand side is 0 for some sequence iff the left hand side is 0 for some sequence as desired.
\end{proof}

Ultraproducts are saturated in the sense that, if $S$ is a countably family of equations so that every finite subset has a solution in $M(\ax)^{\s U}$, then $S$ has a solution in $M(\ax)^{\s U}$. For local statistics this means the following:

\begin{prop}\label{prop:fakesaturation}
    For any $\ax$ and any ultrafilter, if $\nu\in \s L_{rk}(\ax^{\s U})$, then $\nu=\mu^r(f,\ax^{\s U})$ for some $f$.
\end{prop}
\begin{proof}
    By the above, $\nu$ is a limit of measures $\mu^r(f_i,\ax)$ where $f_i\in M(\ax)$. We can let $f=\lim_{i\in \s U} f_i$.
\end{proof}

We can also extract standard actions from ultraproducts. The following is a special case of the Lowenheim--Skolem theorem for continuous logic.

\begin{thm}\label{prop:fakelowenheimskolem}
    For any countable $S\subseteq \lim_{\s U} M(\ax_i)$, there is a standard pmp action $\bx$ and an embedding $\Phi: M(\bx)\rightarrow \lim_{\s U} M(\ax_i)$ so that $\s L_{rk}(\bx)=\lim_{\s U} \s L_{rk}(\ax_i)$ and $S\subseteq \im(\Phi).$
\end{thm}

Roughly, we can take $\bx$ to be a point realization of the closure of $S$. In fact, we can arrange it so that $\Phi$ is much stronger than an embedding. See \cite{LocalGlobalBackground} for a full statement.

So, if $\nu\in \s L_{rk}(\s H)$, then we can swap $\s H$ out for an equivalent hypergraph where $\nu$ is the local statistics of some labelling (rather than just approximated by the statistics of labellings).

\begin{cor}\label{cor:swap}
    For any pmp action $\ax$ and $\nu\in \s L_{rk}(\ax)$, there is local-global equivalent standard pmp action $\bx$ with $\nu=\mu^r(f,\bx)$ for some $f$.
\end{cor}
\begin{proof}
    There is such an $f$ in the ultrapower of $M(\ax)$, so we can take $\bx$ and $\Phi$ as in the previous theorem so that $f\in\im(\Phi)$.
\end{proof}

And, we can characterize local-global containment for actions directly in terms of the dynamics of ultrapowers (this might remind informed readers of weak-containment)

\begin{thm}
    The following are equivalent for a pmp actions $\ax$ and $\bx$:
    \begin{enumerate}
        \item $\bx\preccurlyeq_{lg} \ax$
        \item $\ax^{\s U}$ factors onto $\bx$ via a factor map that preserves stabilizers.
    \end{enumerate}
\end{thm}
\begin{proof}[Proof sketch]
For the implication from (2) to (1), consider any $[k]$-labelling of $\bx$, $f=\sum_{i=1}^k ie_i$. We can read off the local statistics of $f$ from the traces of various algebraic functions of the $e_i$s. A factor maps corresponds to an embedding of $M(\bx)$ into $M(\ax)^{\s U}$, say via $\Phi$, so $\Phi(f)$ has the same local statistics as $f$. This mean $\mu^r(f,\bx)\in \s L_{rk}(\ax^{\s U})$ The previous proposition then tells us $\mu^r(f,\bx)\in \s L_{rk}(\ax)$.

For the other implication, since ultraproducts are saturated, it suffices to simulate any finite subset of $M(\bx)$ up to $\epsilon$.  The local statistics of $\bx$ tell us how to simulate the algebra generated by finitely many characteristics functions. And the characteristic functions generate a dense subalgebra.
\end{proof}

We're predominantly interested in containment, convergence, and equivalence of \textit{hypergraphs}, though. And the correspondence between containment of hypergraphs and containment of their markings is a little messier than one would like. Even in the case of uniformity $2$, all 2-regular acyclic pmp graphs are local-global equivalent, but some of them come from $\Z$-actions while other do not.

Still, we have the following:

\begin{thm}
For pmp hypergraphs $\s G$, $\s H$, the following are equivalent
    \begin{enumerate}
        \item $\s G$ local-global contains $\s H$
        \item For any marking $\bx$ of $\s H$, there is some $\s H(\ax)\equiv_{lg} G$ and so that $\bx \preccurlyeq \ax$
        \item For some marking $\bx$ of $\s H$ and some $\s  H(\ax)\equiv_{lg} \s G$, $\bx\preccurlyeq_{lg} \ax$.
    \end{enumerate}
\end{thm}
\begin{proof}[Proof sketch]
Points (2) and (3) both imply (1) by Proposition \ref{prop:coding}. Of course (2) implies (3). To conclude $(2)$ from $(1)$, note that a marking is coded by its diagram, a certain kind of edge-labelling. Suppose $\s H(\bx)\preccurlyeq_{lg} \s H(\ax)$. Then, in $M(\ax^{\s U})$ we can find a code for a marking, $c$, with all the same local-statistics as $\bx$. By Theorem \ref{prop:fakelowenheimskolem}, there is some standard pmp action $\ax'\equiv \ax$ so that, up to isomorphism $c\in M(\ax')$. Say $c$ induces a marking $\mathsf{c}$ of $\s H(\ax')$. Then $\bx \preccurlyeq_{lg} \mathsf{c}$ and $\s H(\mathsf{c})=\s H(\ax')\equiv \s H(\ax).$
\end{proof}

And, typically, compactness of the space of (weak equivalence classes of) actions obviates this issue of choosing markings. We can conclude from the forgoing that local-global Cauchy sequences of hypergraphs converge to hypergraphs.

\begin{cor}[c.f. Hatami--Lovasz--Szegedi]
    If $\ip{G_i:i\in\N}$ is a sequence of pmp hypergraph which is local-global Cauchy, then $\ip{G_i:i\in\N}$ local-global converges to a local-global equivalence class of some pmp hypergraph.
\end{cor}
\begin{proof}
    It suffices to find a limiting class for some subsequence. Let $\ip{\ax_i:i\in\N}$ be a sequence of markings of $G_i$. We can find a subsequence $f(i)$ so that $\s L_{rk}(\ax_{f(i)})$ converges to $\s L_{rk}(\lim_{\s U} \ax_{i})$ for all $r,k$. Then, by \ref{prop:fakelowenheimskolem}, we find a standard pmp action $\ax$ with $\s L_{rk}(\ax)=\s L_{rk}(\lim_{i\in \s U} \ax_i).$
\end{proof}

We can also conclude compactness for the space of (equivalence classes of) pmp hypergraphs under local-global convergence:

\begin{cor}
    For any sequence of pmp hypergraphs $\ip{\s H_i: i\in\N}$ with uniformly bounded degree, there is a pmp hypergraph $\s H$ and a subsequence of $\ip{\s H_i: i\in\N}$ which converges to $\s H$.
\end{cor}
\begin{proof}
    Since $\bb L(r,k)$ is finite, $P_1(\bb L(r,k))$ is compact, and so is its Hausdorff space $K\bigl(P_1(\bb L(r,k))\bigr)$. So, the family of local-statistics of $\s H$ lives in the compact space $\prod_{r,k} K\bigl(P_1(\bb L(r,k))\bigr)$. So, $\ip{\s H_i:i\in\N}$ must have a local-global Cauchy subsequence. And by the previous theorem, this sequence converges to some pmp hypergraph.
\end{proof}

It follows that any hypergraph parameter expressible in terms of the operations in $M(\ax)$ is continuous with respect to local-global convergence. More generally, any continuous first order sentence is continuous with respect to local-global convergence. For example, the spectral radius of the adjacency operator of a graph is continuous with respect to local-global convergence.

And lastly, we have a lemma for swapping out approximate solutions for exact solutions in equivalent hypergraphs analogous to Corollary \ref{cor:swap}:

\begin{thm}\label{thm:swapgraph}
    For any pmp hypergraph $\s H$ and $\nu\in \s L_{rk}(\s H)$, there is local-global equivalent hypergraph $\s G$ with and a labelling $f$ of $\s G$ so that $\nu=\mu^r(f, \s G)$. 
\end{thm}
\begin{proof}
    Pick a marking for $\s H$ use Corollary \ref{cor:swap} and the coding from Proposition \ref{prop:coding}.
\end{proof}

\subsection{Ab\'ert--Weiss for hypergraphs}

Fundamental to the theory of local-global containment is the existence of a $\preccurlyeq_{lg}$-minimal element among graphs or actions with a given local geometry. (In fact, this minimal element is rather concrete, as we shall see). This is the Ab\'ert--Weiss theorem for actions and the Hatami--Lovasz--Szegedy theorem for graphs. Here we'll establish the same for hypergraphs.

Let's be a little more precise about what I mean by local geometry. We only define weak containment for actions of a fixed group, and the Ab\'ert--Weiss theorem says the Bernoulli shift is a minimal element for free actions. Tucker--Drob generalized this theorem to actions with a fixed stabilizer IRS. In other words, we have a minimal element once we fix the probability of seeing a certain relation among generators. For graphs, Hatami--Lovasz--Szegedy proved that the Bernoulli graphing over a fixed unimodular random graph is minimal. That is, once we fix the probability of seeing a certain graph as a finite neighborhood we can find a minimal element. We'll generalize this latter theorem to hypergraphs.

All of these generalizations follow the same pattern as the Ab\'ert--Weiss theorem. We define the minimal object with a given local geometry by enriching the local geometry with a sequence of iid coin flips at each vertex. Luzin's theorem tells us we can approximate any labelling of this Bernoulli-like object with a map $f$ that only depends on finitely many bits in some finite neighborhood of each vertex. Given any pmp graph (or hypergraph or action), we can take any Borel distance coloring with enough colors and randomly partition the colors to get a $2^n$-labelling, $\ell$, so that the local statistics are close to uniform (over the given local geometry). Then, plugging this labelling $\ell$ into the approximation $f$ gives us a new labelling with local statistics close to our target. 

\begin{dfn}
    For $d,u\in \N$ and compact metric space $X$, the space $\bb L(\infty, X, d,u)$ is the space of isomorphism classes rooted, $X$-labelled, $u$-uniform hypergraphs of degree at most $d$. We will typically suppress $d,u$ in the notation. We endow this space with the topology generated by the open sets \[N_{[G,v,g],n}:=\{[H,o,h]: (B_n(o), o, h)\cong (G,v,g)\}\] for $[G,v,g]\in \bb L(n,X)$. In other words, a sequence of hypergraphs converges if larger and larger neighborhoods of the labelled hypergraphs are isomorphic.

    We write $\bb L(\infty,d,u)$ for the space of isomorphism classes of rooted, unlabelled, $u$-uniform hypergraphs of degree at most $d$ equipped with the analogous topology. Again, we will typically suppress $d$ and $u$.
\end{dfn}

Note that, since we've fixed an upper bound on the degree, this space is compact. The local geometry of a hypergraph $\s H$ tells us how likely we are to see a given element of $\bb L(\infty, d,u)$ as a component in $\s H$.

\begin{dfn}
    For a pmp hypergraph $\s H$ with invariant measure $\mu$, the \textbf{local geometry} of $\s H$ is the measure $\mu_0$ on $\bb L(\infty, d, u)$ where, for $[H,o]\in \bb L(n, d,u)$  any isomorphism class of rooted unlabelled radius $n$ hypergraphs,
    \[\mu_0(N_{[H,o]})=\bb P\bigl( (B_n(x),x)\cong (H,o)\bigr).\]
\end{dfn}

Not all probability measures on $\bb L(\infty)$ arise as local geometries of pmp hypergraphs. Any local geometry will satisfy a mass transport principle inherited from the measure theoretic handshake lemma. For random graphs, this is the unimodularity condition. 

\begin{dfn}
    For any $k,u,d\in \N$, the associated \textbf{labelled Bernoulli hypergraph} is the Borel hypergraph, $\s B=(\bb L(\infty, 2^\N), B)$ where edges correspond to moving the root. That is, (recall than when $H=(X,E)$ is a hypergraph $E$ consists of unordered edges and $\vec H$ is the corresponding set of ordered edges)
    \begin{align*}([H_1,o_1, g_1],...,[H_u,0_u,g_u])\in \vec {\s B}:\lra & (\exists e=(e_1,...,e_u)\in \vec H_1)\; e_1=o_1\mbox{ and }\\ & \quad (\forall i) \;(H_i, o_i, g_i)\cong (H_1, e_i, g_1). \end{align*}

    When $\mu_0$ is the local geometry of some labelled pmp hypergraph, we write $\s B(\mu_0)$ for the the labelled Bernoulli hypergraph viewed as a pmp hypergraph with the invariant measure $\mu$ defined on cylinder sets as follows. For $\s A\subseteq \bb L(\infty,k\times 2^\N)$,
    \[(\mu\ltimes \lambda)(\s A)=\int_{[H,o]} \bb P_{\ell}([H,o,\ell]\in \s A) \;d\mu_0.\]
\end{dfn}
It doesn't matter if we mean that $\ell$ is a vertex or edge or vertex-and-edge labelling in the above definition. Since we have infinitely many independent coin flips at each point, we can spread these around the hypergraph as needed.

There are lots of self loops in $\s B$, but these all have measure $0$ in $\s B(\mu_0)$. Indeed, the set of labelled hypergraphs where each vertex has a unique label is measure $1$. Implicit in this definition is the following fact:

\begin{prop}
    If $\mu_0$ is the local geometry of some labelled pmp hypergraph, then $\s B(\mu_0)$ is a pmp hypergraph.
\end{prop}
The proof of the above is identical to \cite{HLS}. 

We will show that $\s B(\mu_0)$ is local-global minimal among pmp hypergraphs with local geometry $\mu_0$. We need a couple lemmas to turn the proof sketch above into an actual proof. First we'll show any labelling of $\s B(\mu_0)$ is approximated by a local map. Well, actually first we'll introduce some notation.
\begin{dfn}
    For $g: X\rightarrow Y^A$ and $F\subseteq A$, write $g_F$ for the function defined by
    \[g_F(x)=g(x)\res F.\] 
\end{dfn}

\begin{lem}
   For any labelling $\tilde f: \s B(\mu_0)\rightarrow [k]$ and $\epsilon>0$, there are $r, n\in\N$, an approximating labelling ${f:\s B(\mu_0)\rightarrow [k]}$, and a local rule $f_0: \bb L(r,2^n)\rightarrow [k]$ so that 
    \[(\mu_0\ltimes \lambda)(\{x: \tilde f(x)\not=f(x)\})<\epsilon\] and $f$ is given by the local rule
    \[(\forall [H,o,g]\in \s B(\mu_0))\; f([H,o,g])=f_0([B_r(0),o,g_{[0,n]}]) \] 
\end{lem}
\begin{proof}
    Lusin's theorem says that $\tilde f$ is continuous on a measure $(1-\epsilon)$ compact subset of its domain, $E$. There is a continuous extension of $\tilde f\res E$ to all of $\bb L(\infty, 2^\N)$ by Tietze extension, say $f$. Being continuous on $\bb L(\infty, 2^\N, d, u)$ means exactly that $f$ depends only on finitely many bits in some finite neighborhood of the input.
\end{proof}

Second, we can approximate iid variables on any pmp hypergraph.

\begin{lem}
    For any tolerance $\epsilon>0$, radius $r\in\N$, label set $k\in \N$, and pmp hypergraph $\s H$ with local geometry $\mu_0$, there is a $k$-labelling of $\s H$ with local statistics within $\epsilon$ of the uniformly random iid $k$-labelling over $\mu_0$. That is there is $\ell: \s H\rightarrow [k]$ measurable with
    \[\bigl|\mu^r(\ell, \s H)([H,o,g])-\mu_0([H,o])k^{-|\dom(g)|}\bigr|<\epsilon\] for each $[H,o,g]\in \bb L(r,k)$. 
\end{lem}
\begin{proof}
Let $c>>k,r$ and let $\ell_0: \s H\rightarrow [c]$ be a Borel labelling of $\s H$ so that 
\begin{enumerate}
    \item any two vertices within distance $2r$ of each other get different colors 
    \item each color class has measure less than $\epsilon/(du)^{2r}.$
\end{enumerate} Such a labelling exists by standard coloring arguments, and the second point ensure that if $x,y$ are chosen independently at random from $\s H$, then \[\tag{$*$} \bb P((\exists x'\in B_{2r}(x), y'\in B_{2r}(y))\;\ell_0(x')=\ell_0(y'))<\epsilon.\] Let $\mathbf{s}:[c]\rightarrow [k]$ be uniformly random. I claim that, for each $[H,o,g]\in \bb L(r,k)$, 
\[\bb P_{\bf s}(|\mu^r(\bf s\circ \ell_0)([H,o,g])-\mu_0(N_{[H,o]})k^{-|\dom(g)|}|<\epsilon)> O(\epsilon).\] (So, in particular, there is some relabelling $s$ so that $s\circ \ell_0$ will serve as $\ell$ in the theorem). We'll use a standard second moment method argument.

First, let's compute the expected value of $\mu^r(\mathbf{s}\circ\ell_0)([H,o,g])$. Define
\[\bf\chi(x)=\left\{\begin{array}{ll} 1 & (B_r(x),x,\bf s\circ \ell_0)\cong (H,o,g) \\ 0 & \mbox{otherwise} \end{array}\right.\] So,
\[\bb E\bigl(\mu^{r}(\mathbf{s}\circ \ell_0)([H,o,g])\bigr)=\bb E(\int_{x}\mathbf\chi(x)\; d\mu)=\int_{x}\bb E(\mathbf\chi(x)) \; d\mu.\]
And since $\mathbf{s}\circ \ell_0(x)$ is independent of $\mathbf{s}\circ \ell_0(y)$ for $x,y$ within $r$ of each other, for each $x\in \s H$
\[ E(\mathbf\chi(x))= \mu_0([H,o])k^{-|H|}.\]

Second, we bound the variance of $\mu^r(\mathbf{s}\circ \ell_0)([H,o,g])$. We have
\[\bb E\bigl(\mu^{r}(\mathbf{s}\circ \ell_0)([H,o,g])^2\bigr)=\bb E(\int_{(x,y)}\mathbf\chi(x)\chi(y) d(\mu\times \mu))=\int_{(x,y)}\bb E(\mathbf\chi(x)\mathbf\chi(y)) d(\mu\times \mu).\] By ($*$), $\mathbf{s}\circ \ell_0(x)$ and $\mathbf{s}\circ \ell_0(y)$ are independent at all points within $r$ of $x$ or $y$ for a ${(1-\epsilon)}$-fraction of points in $X^2$. So, this works out $(1-\epsilon)(\mu_0([H,o])k^{-|\dom(g)|})^2+O(\epsilon).$ Thus, the variance is $O(\epsilon).$ Thus, theorem follows by Chebyshev's inequality.
\end{proof} 

Now, we can combine the previous two lemmas.

\begin{thm}
    If $\s H$ is a pmp hypergraph with local geometry $\mu_0$, then the associated Bernoulli hypergraph is local-global contained in $\s H$: $\s B(\mu_0)\preccurlyeq_{lg} \s H$.
\end{thm}
\begin{proof}
Let $\s H$ be any pmp hypergraph with invariant measure $\mu$. Fix any $k$-labelling $f$ of $\s B(\mu_0)$, radius $r_0$, and $\epsilon>0$. We want to approximate the local statistics of $f$ with a labelling of $\s H$. The idea is clear: plug a random labelling of $\s H$ into a local approximation of $f$. Let's check this works.

For $f_0: \bb L(r,2^n)\rightarrow [k]$, define $\Phi_{f_0}$ to be the precomposition map:
\[\Phi_{f_0}:\bb L(r+r_0, 2^n)\rightarrow \bb L(r_0,k)\]
\[[H,o,g]\mapsto[B_{r_0}(o),o,f_0\circ g].\] Note that we need to know $g$ to radius $r+r_0$ to learn $f_0\circ g$ to radius $r_0$.

Let $\ell$ be the canonical iid $2^\N$ labelling of $\s B(\mu_0)$. So, for a vertex $[H,o,g]\in \s B(\mu_0)$, 
   \[\tilde \ell([H,o,g])=g(o)\] and for an edge $e=([H,o_1,g],...,[H,o_u,g])$,
   \[\tilde \ell(e)=g(o_1,...,o_u).\] The approximation lemma above tells us we can find $r,n\in\N$ and $f_0: \bb L(r, 2^n)\rightarrow [k]$ so that, for any $[H,o,g]\in \bb L(r_0,k)$,
\begin{align*}
    \mu^{r_0}(f, \s B(\mu_0)) & =(1+O(\epsilon))\mu^{r_0}(f_0\circ \ell_n,\s B(\mu_0))([H,o,g])\\
    &= (1+O(\epsilon)) \mu^{r+r_0}(\ell_n, \s B(\mu_0))(\Phi_{f_0}\inv([H,o,g]))\\
    & = (1+O(\epsilon)) \sum_{[H',o',g']\in \Phi_{f_0}\inv[H,o,g]} \mu_0([H',o'])k^{-|\dom(g')|}.
\end{align*}
The lemma just above says that we can find some labelling $\ell_0$ of $\s H$ so that, for $[H',o',g']$,
\[\mu^{r_0+r}(\ell_0, \s H)([H',o',g'])=\bigl(1+O(\epsilon)\bigr)\mu_0([H',o'])k^{-|\dom(g')|}.\] So,
\begin{align*}
    \mu^{r_0}(f_0\circ \ell_0, \s H) & = \mu^{r+r_0}(\ell_0, \s H)(\Phi_{f_0}\inv([H,o,g]))\\
    & =\bigl(1+O(\epsilon)\bigr)\sum_{[H',o',g']\in \Phi_{f_0}\inv[H,o,g]}\mu_0([H',o'])k^{-|\dom(g')|}\\
    & = \bigl(1+O(\epsilon)\bigr)^2\mu^{r_0}(f, \s B(\mu_0))
\end{align*} Thus, $f_0\circ \ell_0$ approximates $f$ as desired.
\end{proof}

The utility of these theorems depends on how easy it is to reason about $\s B(\mu_0).$ Luckily, this is pretty straightforward. In the case where $\mu_0$ is a Dirac mass concentrated on a vertex transitive graph $G$, measurable labellings of $\s B(\mu_0)$ are the same as factor of iid labellings of $G$. More generally, a Borel map on $\bb L(\infty, 2^\N)$ corresponds to a global family of automorphism equivariant random variables on countable graphs.

\begin{dfn}
    A family of $X$-valued random variables $\ip{\bf f_G: G\in \bb L(\infty, d, u)}$ with $\bf f_G$ of domain $[0,1]^G$, is \textbf{automorphism equivariant} if 
    \[[(G,o,\ell)]= [(H,v,g)]\rightarrow [(G,o, \bf f_G(\ell))]=[(H,v, f_H(g))]\]

    We call such a family Borel if
    \[[G,v,\ell]\mapsto [G,v,\bf f_G(\ell)] \] is a Borel map from $\bb L(\infty, d,u,[0,1])$ to $\bb L(\infty, d, u, X)$.
\end{dfn}

The following proposition amounts to Currying some variables:

\begin{prop}
    For a Borel $X$-labelling of $\s B(\mu_0)$ $F$, the family of random variables $\ip{\bf f_G: G\in \bb L(\infty, d,u)}$ given by
    \[\bf f_G(\ell)(v)=F([G,v,\ell])\] is Borel and automorphism equivariant. And every automorphism equivariant Borel family of random variables is associated to a labelling of $\s B(\mu_0)$ like this.
\end{prop}

Suppose $F$ induced a family of random variables $\bf f_G$ as above. By definition, the probability of seeing some label $a$ is at a vertex is:
\[\mu(\{[G,r,l]: F([G,r,l])=a\})=\int_{[G,v]} \bb P(\bf f_G(\ell)(r)=a)\;d\mu_0.\]

More generally, for some event $\s A$ in $\bb L(\infty,d,u,X)$, if we write $\tilde G$ for the component of $[G,r,\ell]$ in $\s B(\mu_0)$ and $\tilde r$ for $[G,r,\ell]$, then 
\[\mu(\{[G,r,\ell]: [\tilde G, \tilde r, F]\in \s A\})=\int_{[G,r]} \bb P([G,r,\bf f_{G}(\ell)]\in \s A)\;d\mu_0.\] In particular, we can bound this above and below by probability $\bf f_G(\ell)$ produces an $\s A$-labelling given the constraints of the local geometry.

\begin{dfn}
    For a measure $\mu$, write $\supp(\mu)$ for its \textbf{support}. In particular, for a local geometry $\mu_0$, $\supp(\mu_0)$ is the set of hypergraphs whose neighborhoods all appear with positive probability in $\mu_0$.
\end{dfn}

\begin{prop}
    For any event $\s A$ in $\bb L(\infty,d,u, X)$, 
    \[\min_{[G,r]\in \supp(\mu_0)} \bb P([G,r,\bf f_G(\ell)]\in \s A)\leq \mu(\{[G,r,\ell]: [\tilde G, \tilde r, F]\in \s A\})\] and \[\mu(\{[G,r,\ell]: [\tilde G, \tilde r, F]\in \s A\}) \leq \max_{[G,r]\in \supp(\mu_0)} \bb P([G,r,\bf f_G(\ell)]\in \s A) \]
\end{prop}

As a simple example, we can show that if $\mu_0$ is any $d$-regular local geometry on graphs ($u=2$), then $\s B(\mu_0)$ has a measurable independent set of measure at least $1/(4d)$: Let $\bf f_G$ be $\{0,1\}$ valued with $\bf f_G(v)=1$ iid for every $v$ with probability $1/d$, then define our independent set by $v\in A$ if $\bf f_G(v)=1$ and for any neighbor of $u$ of $v$, $\bf f_G(u)=0$. For any $[G,v]$ in the support of $\mu$, $v$ has exactly $d$ neighbors, so the probability $v$ is $A$ is $\frac{1}{d}(1-\frac{1}{d})^d\geq \frac{1}{4d}.$ 

\section{Trivial CSPs}

Our first application is to the theory of constraint satisfaction problems, or CSPs. The CSP associated to a finite relational structure $\s D$ is the following problem: given a structure $\s X$ in the same signature as $\s D$, determine if there is a homomorphism from $\s X$ to $\s D$. In this context, we call $\s X$ an instance of $\s D$. For a fixed $\s D$, this gives a rich family of problems with a beautiful intrinsic theory of complexity.

In descriptive set theory, we often want to understand versions of CSPs where the instance $\s X$ is definable-- Borel, measurable, OD, etc. In \cite{csps}, the author determines which are the structures $\s D$ where the Borel CSP is no more complicated than the classical CSP:

\begin{thm}[{\cite[Theorems 5.2, 5.7]{csps}}]\label{thm:thesisthm}
    For a finite relational structure $\s D$, the following are equivalent:
    \begin{enumerate}
        \item $\s D$ is width-1
        \item For any Borel $\s X$, there is Borel homomorphism from $\s X$ to $\s D$ iff there is a homomorphism from $\s X$ to $\s D$.
    \end{enumerate}
\end{thm} There is a deep literature on width-1 structures in computer science, and they can be defined in a number of ways. A combinatorial definition is that $\s D$ is with-1 if an instance $\s X$ is solvable exactly when every acyclic lift of $\s X$ is solvable. Lemma \ref{lem:acyclicdoodad} gives a slightly more detailed version this. The theorem above says we think of these structures as being trivial from the point of view of Borel combinatorics.

We can use the limit theory for hypergraphs to sharpen this result in two ways. First, we can add to this list of equivalences that $\s D$ is measure theoretically trivial.  Second, we can extract from this theorem a purely finitary consequence giving yet another characterization of width-1 structures.

\subsection{Measurable CSPs}
Recall that a pmp structure is a family of pmp relations on a standard measure space, that is, a family of relations which only contain edges between $E$-related vertices for some pmp equivalence relation $E$. We can define local-global statistics, convergence, and equivalence exactly as for hypergraphs. And, we can define ultraproducts using generating maps to code the relations. Compactness of the space local-global convergence follows as before. We are interested in measurable CSPs. 

\begin{dfn}
    For $\s D$ a finite relational structure and $\s X$ a pmp structure in the same signature, a \textbf{measurable solution} is a measurable map which is a homomorphism when restricted to a measure 1 set. 
\end{dfn}

Since Borel maps are always measurable and any pmp structure differs from a Borel structure by measure 0, the following is a simple corollary of Theorem \ref{thm:thesisthm}.

\begin{cor}\label{cor:width1impliestrivial}
    If $\s D$ is width-1, every pmp $\s X$ with a homomorphism to $\s D$ has a measurable homomorphism to $\s D$.
\end{cor}

We could also prove this directly (a little more simply than in the Borel case where degrees may be uncountable). For readers familiar with width-1 structures, we can carry out the usual arc-consistence algorithm in a measurable way because the Luzin--Novikov theorem gives us a measurable way of listing out the relations each point sees.

\subsection{When does solvable imply measurably solvable?}

The converse of Corollary \ref{cor:width1impliestrivial} just above is also true. To prove this, we'll use a lemma from the proof of Theorem \ref{thm:thesisthm} and a local-global limit construction.

Roughly, the implication from (2) to (1) in Theorem \ref{thm:thesisthm} goes as follows: $\s D$ is width-1 means that any instance $\s X$ of $\s D$ has a solution iff any acyclic lift of $\s X$ has a solution. In \cite{csps}, this is sharpened to say that if $\s D$ is not width-1, then there is there is an unsolvable instance $\s X$ so that any acyclic lift has a solution, but there is a vertex $v\in X$ and an acyclic lift $\s Y$ which has no solution which is constant on the fibre of $v$. Then by gluing copies of $\s Y$ together along an acyclic hypergraph, we always end up with a lift of $\s X$, so with a structure that has a solution. But if we use a Borel acyclic hypergraph with very high Borel chromatic number the result will not have a Borel solution. Such hypergraphs can be found using Baire category arguments.

So, to prove the measurable analogue, we need an acyclic pmp hypergraph with arbitrarily high measurable chromatic number. It turns out we can take a local-global limit of a high enough degree random regular hypergraph. The next lemma captures the combinatorial heart of the Borel argument:

\begin{lem}[{\cite[Lemma 5.6]{csps}}]\label{lem:acyclicdoodad}
    If $\s D$ is not width-1 there is a relation $R$ on $D$ and which has the following properties:
    \begin{enumerate}
        \item There is some conjunction of relations from $\s D$, $S$ so that \[R(x_1,...,x_n)=(\exists z_1,...,z_n) S(z_1,...,z_n, x_1,...,x_n)\] (to use the jargon, $R$ is pp-definable in $\s D$)
        \item any acyclic instance of $(D,R)$ has a solution
        \item there is no constant sequence in $R$: for all $x$, $\neg R(x,x,x,...,x)$
    \end{enumerate}
\end{lem}

Point (1) says that we can view any instance of $(D, R)$ as an instance of $\s D$ by replacing each $R$-edge with a constellation of $\s D$ relations and dummy variables. 

Now, we'll invoke some facts about random regular hypergraphs. First, they are asymptotically acyclic:

\begin{lem} \label{lem:asymptoticallyacyclic}
    If $\ip{\s H_n:n\in\N}$ is a sequence of $d$-regular $u$-uniform hypergraphs with $|\s H_n|=c_n\geq n$ and $\s H_n$ chosen uniformly at random from all $d$-regular $u$-uniform hypergraphs of size $c_n$, then, with probability $1$
    \[\lim_{n\in\N} \operatorname{girth}(\s H_n)= \infty\]
\end{lem} The proof is essentially the same as for graphs. For instance, see \cite[Lemma 24]{DumZhu}.

We'll also use the following bound due to Bennett and Frieze (One can also prove this with an entropy inequality):
 \begin{thm}[\cite{BF}]
     If $\ip{\s H_n:n\in\N}$ is a growing sequence of uniformly random $d$-regular $u$-uniform hypergraphs, then almost surely, 
     \[\limsup_{n\in\N} \alpha_{\mu}(\s H_n)\leq O\bigl(( \log(d)/d)^{1/(u-1)}\bigr).\]
 \end{thm}

 Really, we just need that as $d$ goes to infinity, the independence ratio goes to $0$. Putting these together we get the following:
 
\begin{thm}\label{thm:trivialmsrblcsp}
    For a finite relational structure $\s D$, if $\s D$ is not width-1, then there is a pmp instance of $\s D$, $\s X$, so that $\s X$ has a solution but no measurable solution.\end{thm}
\begin{proof} Suppose $\s D$ is not width-1, and fix $R$ as in Lemma  \ref{lem:acyclicdoodad}. let $u$ be the arity of $R$, let $d$ be large enough that a random $d$-regular $u$-uniform hypergraph has no independent set of density more than $\frac{1}{2|D|}$, and let $\ip{\s H_n:n\in\N}$ be sequence of random $d$-regular $u$-uniform hypergraphs.

    Now we'll define a pmp instance of $\s D$. By point (1) of Lemma \ref{lem:asymptoticallyacyclic}, it suffices to build an instance of $(D, R)$. Let $\s H$ be a local-global limit of any subsequence of $\ip{\s H_n:n\in\N}$. One exists by compactness, and the independence ratio of $\s H$ is at most $\frac{1}{2|\s D|}$. So map from $\s H$ to $\s D$ is constant on some hyperedge.  Now let $\s X$ have domain the same as $\s H$. Choose a Borel ordering $\prec$ on the set of vertices. Define $R^{\s X}$ by 
    \[R^{\s X}(x_1,...,x_u):\lra (x_1,...,x_u)\in \s H\mbox{ and }x_1\prec x_2\prec...\prec x_u.\]

    By Lemma \ref{lem:asymptoticallyacyclic}, $\s H$ is acyclic, so $\s X$ is an acyclic instance of $(D,R)$. Thus $\s X$ has a (set theoretic) solution. But, if $f: \s X\rightarrow \s D$ is any measurable map, $f$ is constant on some hyperedge in $\s H$, $e=\{x_1,...,x_u\}$. Then (up to reordering), $R^{\s X}(x_1,...,x_u)$. But since $R$ doesn't contain a constant sequence, $\neg R(f(x_1),f(x_2),...,f(x_u)).$ So $f$ is not a solution.    
\end{proof}
\begin{cor}
    For any finite relational structure $\s D$, $\s D$ is width-1 if and only if every solution to an instance of $\s D$ can be turned into a measurable solution.
\end{cor}

\subsection{Finitary corollary}

We can also run the limit machinery in the other direction to extract a finitary corollary from Corollary \ref{cor:width1impliestrivial}. Measurable solutions on limit structures correspond to asymptotic solutions along the sequence. So, Corollary \ref{cor:width1impliestrivial} should tell us that the trivial CSPs are those where, if finite obstructions to a solution disappear in the sequence, the sequence has an asymptotic solution. 

\begin{dfn}
    For a relational structure $\s D$ and instance $\s X$ of $\s D$, the obstruction number is
    \[\ob_{\s D}(\s X)=\sup\bigl\{|A|: A\subseteq X, \;\s X\res A\mbox{ has no solution}\bigr\}.\] 

    For a pmp instance $\s X$ of $\s D$ with invariant measure $\mu$, the solution density is the optimal fraction of relations we can satisfy:
    \[\rho_{\s D}(\s X)=\inf_{R\in \tau} \sup\bigl\{\mu(f\inv[R])/\mu(R^{\s X}): f: \s X\rightarrow \s D\mbox{ measurable}\bigr\}.\]
    
    If $\ip{\s X_i:i\in\N}$ is a sequence of pmp instances of $\s D$, then we say that $\ip{\s X_i:i\in\N}$ has an asymptotic solution if it's solution densities tend to one, i.e.
    \[\lim_{i\in \N} \rho_{\s D}(\s X_i)=1.\]
\end{dfn}

We then have

\begin{thm}
    The following are equivalent for a finite relational structure $\s D$:
    \begin{enumerate}
        \item $\s D$ is width-1
        \item If $\ip{\s X_i: i\in\N}$ is a sequence of instances with bounded degree, then $\ip{\s X_i:i\in\N}$ has an asymptotic solution iff $\lim_{i\in\N} \ob_{\s D}(\s X_i)=\infty$
    \end{enumerate}
\end{thm}

We could strengthen this further by asking that the density of small obstructions vanish rather than the number.

\begin{proof}
    If $\s D$ is not width-1, then gluing instances of $(D,R)$ together along along a random regular hypergraph as in the proof of Theorem \ref{thm:trivialmsrblcsp} gives a counter-example to property (2) above.

    Conversely, if $\s D$ is width-1, consider any sequence $\ip{\s X_i:i\in\N}$ of instances with bounded degree with $\ob_{\s D}(\s X_i)$ tending to infinity. Suppose toward contradiction that $\limsup_i \rho_{\s D}(\s X_i)<1$. Then let $\s X$ be any subsequential local-global limit (one exists by compactness). By the definition of local-global limit, \[\ob_{\s D}(\s X)=\lim_i \ob_{\s D}(\s X_{f(i)})=\infty\]
    and \[\rho_{\s D}(\s X)=\lim_i \rho_{\s D}(\s X_{f(i)})\leq \limsup_i \rho_{\s D}(\s X_{i})<1 .\] But the first displayed equation and Theorem \ref{thm:trivialmsrblcsp} together tell us that $\s X$ has a measurable solution, which is contradicts the second displayed equation.
\end{proof}

\section{Nibble nibble}

Our second application is about reinforcing the bridge between probabilistic and descriptive combinatorics. In many probabilistic arguments-- such as the so-called R\"odl nibble or differential equation methods-- one builds a solution to a problem in a sequence of small random steps. Typically an easy computation shows that a small random step moves you towards a solution in expectation, then a concentration argument of some kind says you can find a trajectory that gets you all the way (or 99\% of the way) to a solution. 

These arguments can also be run on pmp structures, and it turns out that some of the annoying technicalities disappear: the expected change becomes a literal change in measures, regularity assumptions can be bootstrapped, the process becomes more symmetric since we don't have to processes single vertices at a time, etc. We'll illustrate this softening phenomenon with two arguments. First, we'll prove a version of the R\"odl--Frankl matching theorem for pmp hypergraphs, then we'll give another matching argument for hypergraphs of large girth with explicit bounds on the size of the matching.

\subsection{Frankl--R\"odl matching}

 A matching in a hypergraph is a collection of disjoint hyperedges. The Frankl--R\"odl theorem says that every hypergraph with degrees mostly close to $\Delta$ and most pairs of vertices having small codegree (i.e.~most vertices has few edges in common) has a matching that covers a $(1+o(\Delta))$-fraction of the vertices \cite{FranklRodl}. For graphs (where codegree is always at most $1$), this is a weakening of Vizing's theorem. The proof is considered straightforward nowadays: pick a small random set of edges (a nibble) to move from the graph into a matching; compute how this affects degrees and codegrees; and iterate. A few second moment arguments are used to show that there is one set of edges whose removal has an effect close to expectation. This is the prototypical nibble argument.

We'll start by proving a measurable version of the Frankl--R\"odl matching theorem, then we'll show how to use it to recover the classical theorem.  In our measurable proof, we'll see that most of the second moment arguments get hidden inside the Ab\'ert--Weiss theorem and that we can simplify the regularity assumptions in proving the key lemma. 

In case the reader is unfamiliar with the classical theorem, here is a more detailed overview of our proof. The basic steps we'll take (our nibbles) are to pick a small set of edges, $C$, at random, throw all the edges in $C$ which don't overlap with other $C$ edges into our matching, and delete all the $C$ dedges and their vertices.

\begin{dfn}
    For $\s H$ a pmp hypergraph, and $C$ (for change) a measurable set of edges, let
    \[\s H^+(C)=(V^+(C), H^+(C))\] where
    \[V^+(C)=\{v\in V: (\forall e\in C)\; v\not\in e\}\]
    \[M(C)=\{e\in H: (\forall e\in f)\; e\cap f=e\mbox{ or } \emptyset\}.\]
    \[H^{+}(C)=H\cap [V^+(C)]^u\]
    \[\deg^+_C(v):=|\{e: e\cap \bigcup C=\emptyset, v\in e\}|.\]
\end{dfn}

Back of the envelope calculations suggests that if $\s H$ is close to $D$-regular and has small codegree, and if $C$ is iid distributed with probability $\epsilon/D$, then $V^+(C)$ has measure $e^{-\epsilon}$, and we cover a $(1-e^{-\epsilon})(1-O(\epsilon D))$-fraction of vertices with our matching. So, if we iterate these small random steps with small enough $\epsilon$ around $-\log(\epsilon)/\epsilon$ many times, we should end up covering most of the vertices. These calculations are all approximate, so we need to be a little careful with our error.

Fix a tolerance $\delta$ and a degree $\Delta$. We'll call a pmp hypergraph $\delta$-good for $\Delta$ if the probability that the degree of a vertex is far from $\Delta$ and the probability that the codegree of any two vertices is large relative to $\Delta$ are both tolerably low: 

\begin{dfn}
    Fix $\delta, \Delta>0$. We call a pmp hypergraph $\s H$ $\delta$-\textbf{good} for $\Delta$ if
    \[\bb P(|\deg(v)-\Delta|>\delta\Delta)<\delta\] and 
    \[\bb P\bigl(|(\exists w)\;\operatorname{codeg}(v,w)|>\delta\Delta\bigr)<\delta.\]
\end{dfn}
Note that goodness is a property of the local geometry of a hypergraph.

A little more scratch-work suggests that if $\s H$ is $\delta'$-good for $\Delta$ and $C$ is iid distributed with probability $\epsilon/\Delta$, then there is a worse tolerance $\delta$ so that $\s H^+(C)$ is $\delta$-good for $\Delta e^{-\epsilon (u-1)}$. The following lemma makes this precise and computes how many vertices are covered by $M(C)$:

\begin{lem}
    For any tolerance $\delta'$, for any large enough (depending only on $\delta'$) degree $\Delta$, there is some $\delta$ so that for any $\epsilon\in [0,1]$, any atomless pmp hypergraph $\s H$ which is $\delta$-good for $\Delta$ has some change $C$ so that
    \begin{enumerate}
        \item $\s H^+(C)$ is $\delta'$-good for $\Delta e^{-\epsilon(u-1)}$ (with the normalized measure inherited from $\s H$.
        \item $\mu\bigl(\s V^+(C)\bigr)=e^{-\epsilon}\bigl(1+O(\delta')\bigr)$
        \item $\bb P\bigl (\exists e\in M(C))\; v\in e \,|\, v\not\in V^+(C)\bigr)= e^{-\epsilon u}\bigl(1+O(\epsilon)\bigr)\bigl(1+O(\delta')\bigr)$
    \end{enumerate} where the hidden constants only depend on $u$.
\end{lem}
\begin{proof}
    Fix $\delta'>0$. It suffices to find $\delta'$-good changes under the assumption that
    \[\tag{$*$} \bb P\bigl(|\deg(v)-\Delta|>\delta\Delta\bigr)=0 \quad\mbox{ and }\quad \bb P\bigl(|(\exists w)\;\operatorname{codeg}(v,w)|>\delta\Delta\bigr)=0.\]
    This is because having a $\delta'$-good change describes an open set in the local statistics and so we can apply the usual compactness argument: if for every $\epsilon$ there are pmp hypergraphs satisfying 
    \[\bb P\bigl(|\deg(v)-\Delta|>\delta\Delta\bigr)<\epsilon\quad\mbox{ and }\quad \bb P\bigl(|(\exists w)\;\operatorname{codeg}(v,w)|>\delta\Delta\bigr)<\epsilon\]
   with no $\delta'$-good change, then some subsequential limit would have 
   \[\bb P\bigl(|\deg(v)-\Delta|>\delta\Delta\bigr)=0\quad \mbox{ and }\quad \bb P\bigl(|(\exists w)\;\operatorname{codeg}(v,w)|>\delta\Delta\bigr)=0\]
    and no $\delta'$-good change.

    By the hypergraph Ab\'ert--Weiss theorem, it suffices to find a $\delta'$-good change in any good enough Bernoulli hypergraphing. So, let $\s H$ be a Bernoulli hypergraphing. Pick some small enough $\delta$ (to be determined later) and $\Delta>1/\delta$ large enough that $(1-\frac{\epsilon}{\Delta})^\Delta=e^{-\epsilon}\bigl(1+O(\delta)\bigr)$ for all $\epsilon\in[0,1]$. Suppose the local geometry $\mu_0$ satisfies ($*$). Recall, we need to define a isomorphism equivariant family of random labellings, one for each $[H,o]\in \supp(\mu_0)$.
    
    Define a random change $\bf{c}$ (bold to emphasize randomness) as follows: for every $(H,v)$, put each edge in $H$ into $\bf{c}$ independently with probability $\dfrac{\epsilon}{\Delta}$. $H^+(\bf c)$ satisfies the small codegree (with tolerance $\delta e^{\epsilon(u-1)}$) and bounded degree (with the same degree bound) conditions automatically from the same bounds in $H$. We can bound $\mu(V^+(\bf{c}))$, $\bb P(v\mbox{ is covered by }M(\bf{c}))$, and $\bb P\bigl( |\deg^+_\bf{c}(v)-e^{-\epsilon u}\Delta|>\delta'\Delta \bigr)$ by finding the best and worst case for these probabilities over $[H,v]\in \supp(\mu_0)$. Abuse notation slightly and write $H$ also for the set of unordered edges in $H$.

    First, let's compute the probability a vertex isn't deleted. A vertex is deleted exactly when it's contained in some edge in $\bf{c}$. For a vertex $v$, write $E(v)$ for the set of edges containing $v$. For $[H, v]\in \supp(\mu_0)$, $|E(v)|\in [\Delta(1-\delta), \Delta(1+\delta)]$, so

    \begin{align*}
        \bb P_{\ell} \bigl(v\in V^+(\bf{c})\bigr)&= \bb P_\ell (\bigwedge_{e\in E(v)} \;e\not\in \bf{c})\\
        & =  \prod_{v\in e} \bb P_{\ell}(e\not\in \bf{c}) \\
        & = (1-\frac{\epsilon}{\Delta})^{|E(v)|}\\
        & = (1-\frac{\epsilon}{\Delta})^{\Delta(1+O(\delta))} \\
        & = e^{-\epsilon (1+O(\delta))}=e^{-\epsilon}\bigl(1+O(\delta)\bigr).
    \end{align*}


    Now we'll bound the probability a vertex in $\s B(\mu_0)$ has degree far from $\Delta e^{-\epsilon (u-1)}.$ Again, it suffices to bound this probability for each $[H,v]\in\supp(\mu_0)$, which we can do with a second moment argument.
    
    Suppose that $(H,v)\in \supp(\mu_0)$. The probability that an edge $e\in H$ remains in $H^+(\bf{c})$ is just the probability that no edge it touches is in $\bf{c}$, and the number of edges $e$ touches should be about $u\Delta$. More formally define $N(e)$ to be the set of edges $e$ intersects:
    \[ N(e):=\{f\in H: f\cap e\not=\emptyset\}.\] For any $e$, 
    \[|N(e)|=\sum_{v\in e}\deg(v) -O(\sum_{u,v\in e} \cod(u,v))=u\Delta(1+O(\delta)).\] The edge $e$ remains in $H^+(\bf{c})$ when $N(e)\cap \bf{c}=\emptyset$, so
    \begin{align*}
        \bb P_\ell(e\in H^+(\bf{c})) & = \bb P_{\ell}(\bigwedge_{f\in N(e)} \;f\not\in \bf{c})\\
            & = \bigl(1-\frac{\epsilon}{\Delta}\bigr)^{\Delta u(1+O(\delta))}\\
            &= e^{-\epsilon u}\bigl(1+O(\delta)\bigr).
    \end{align*}
    Since $e\in H^+(\bf{c})$ implies $v\in V^+(\bf{c})$ whenever $v\in e$, the expected degree of $v$ conditioned on $v\in V^+(\bf{c})$ is

    \begin{align*}
        \bb E_{\ell}\bigl( \deg^+_\bf{c}(v) \;| \;v\in V^+(e)\bigr)& =\sum_{e\in E(v)} \bb P\bigl(e\in H^+(\bf{c})\; | \; v\in V^+(\bf{c})\bigr) \\
        & = \sum_{e\in E(v)} \frac{\bb P\bigl(e\in H^+(\bf{c})\bigr)}{\bb P\bigl(v\in V^+(\bf{c})\bigr)}\\
        & = \Delta e^{-u \epsilon-\epsilon}\bigl(1+O(\delta)\bigr)=\Delta e^{-\epsilon(u-1)}\bigl(1+O(\delta)\bigr). 
    \end{align*}
    Now we'll compute the variance of the expected degree. As usual, we can write $\deg^+_\bf{c}(v)$ as a sum of indicator functions:
    \[\deg^+_\bf{c}(v)=\sum_{e\in E(v)} \bf 1_{H^+(\bf{c})}(e).\] For the sake of readability, write $\var^v$ $\cov^v$, and $\bb E^v$ for the variance, covariance, and expectation of random variables conditioned on $v\in V^+(\bf{c})$, and write $\bf 1^+_H$ for $\bf 1_{H^+(\bf{c})}$. We then have
    \[\var^v\bigl(\deg^+_\bf{c}(v)\bigr)\leq \bb E^v\bigl(\deg^+_\bf{c}(v)\bigr)+\sum_{e\not=f\in E(v)} \cov^v\bigl(\bf 1^+_H(e), \bf 1+_H(f)\bigr). \] For any $e,f$, $|N(e)\cap N(f)|\leq \sum_{v\in e}\sum_{u\in f} \cod(v,u)\leq u^2 \delta\Delta$. So, the covariance is
    
    \begin{align*}\cov^v(\bf 1^+_H(e), \bf 1^+_H(f)) & = \bb E^v(\bf 1^+_H(e) \bf 1^+_H(f))-\bb E^v(\bf 1^+_H(e))\bb P(\bf 1^+_H(f))\\
    & = \bigl(1-\frac{\epsilon}{D}\bigr)^{|N(e)\cup N(f)\sm E(v)|}- (1-\frac{\epsilon}{D})^{|N(e)|+|N(f)|-2|E(v)|}\\
    & \leq \bigl(1+O(\delta)\bigr)e^{-2(u-1)\epsilon}\bigl((1-\frac{\epsilon}{\Delta})^{-|N(e)\cap N(f)|}-1 \bigr)\\
    & \leq \bigl(1+O(\delta)\bigr)e^{-2(u-1)\epsilon} (e^{-u^2 \epsilon\delta}-1)\\
    & \leq O(\delta).
    \end{align*} So,
    \[\var^v\bigl(\deg^+_\bf{c}(v)\bigr) \leq e^{-\epsilon(u-1)}\Delta+\Delta^2 O(\delta)\]
    Now, by Chebyshev's inequality (and the fact we chose $\Delta>1/\delta)$, 
    \begin{align*}
        \bb P^v\bigl(|\deg^+_\bf{c}(v)-e^{-\epsilon (u-1)}\Delta|>\delta' e^{-\epsilon(u-1)}\Delta\bigr) & \leq \frac{\var^v\bigl(\deg^+_\bf{c}(v)\bigr)\bigl(1+O(\delta)\bigr)}{\bigl(\delta' e^{-(u-1)\epsilon}\Delta\bigr)^2} \\
        & \leq \bigl(1+O(\delta)\bigr)\frac{O(\Delta+\Delta^2\delta)}{O(\delta'\Delta^2)}\\
        & \leq \bigl(1+O(\delta)\bigr)O\bigl(\frac{\delta}{\delta'}\bigr).
    \end{align*}
    
    Finally, we want to compute the probability a vertex is covered given that it's deleted. Since every covered vertex is deleted, it's enough to compute the probability that it's covered. The events that different edges cover $v$ are disjoint, so by the same estimate of $N(e)$ as above,

    \begin{align*}
        \bb P(v\mbox{ is covered by }M(\bf{c})) & = \bb P(\bigvee_{e\in E(v)} N(e)\cap \bf{c}=\{e\}) \\
        & = \sum_{e\in E(v)} \frac{\epsilon}{\Delta} (1-\frac{\epsilon}{\Delta})^{|N(e)|-1}\\
        & = \epsilon e^{\epsilon u}(1+O(\delta))
    \end{align*}

So, the probability a vertex is covered given that it's deleted is 

\[(1+O(\delta)) \frac{\epsilon e^{-\epsilon u}}{1-\epsilon^{-\epsilon}}=(1+O(\delta))(1+O(\epsilon))\]

Finally, we can pick $\delta$-small and $\Delta$ big enough so that all of these $O(\delta)$'s are less than $(\delta')^2$.

    
\end{proof}

Now we can string these together to get a matching that covers most vertices by recursively finding a small enough starting tolerance to get our desired end tolerance.

\begin{thm}
    Fix $\epsilon_0>0$. There is a tolerance $\delta$ so that for all large enough degrees $D$, if $\s H$ is an atomless pmp hypergraph which is $\delta$-good for $D$ then $\s H$ has a matching that covers a set of vertices of measure $1-\epsilon$.
\end{thm}
\begin{proof}
    Fix a target $\epsilon_0$. Choose $\delta$ and $\epsilon$ very small (we'll see how small exactly in a moment). Define $\delta_i$ and $\Delta_i$ by induction as follows:
    \begin{itemize}
        \item $\delta_0=\delta$ and $\Delta_0=1$
        \item Given $\delta_i$ and $\Delta_i$, pick any $\delta_{i+1}$ small enough and $\Delta_{i+1}$ large enough that the pevious lemma applies, i.e.~whenever $\s H$ is $\delta_{i+1}$-good enough for $\Delta_{i+1}$ there is $C$ so that $\s H^+(C)$ is $\delta_i$-good for $e^{-(u-1)\epsilon}\Delta_{i+1}$.
    \end{itemize}
     Let $t$ be large enough that $e^{-\epsilon t}<\epsilon_0$. Let $D$ be large enough that $e^{-\epsilon(u-1)i} D>\Delta_{t-i}$ for each $i$.

    Let $\s H$ be $\delta_t$-good for $D$. By the lemma we have changes $C_0,...,C_t$ so that, if $\s H_i$ is defined inductively $\s H_0=\s H$ and $\s H_{i+1}=\s H_i^+(C_i)$, then
    \[\mu(V_t)\leq e^{-\epsilon t}(1+\delta)^t\]
    \[\mu(\pi M)\geq e^{-u\epsilon}(1-O(\delta))(1-O(\epsilon))\mu(V\sm V_t) \]

    We can choose $\epsilon$ small enough that $e^{-u\epsilon}(1-O(\epsilon))>1-\epsilon_0$. Then choose $t$ large enough that $e^{-\epsilon t}<\epsilon_0$. Then, choose $\delta$ so that $(1+\delta)^t<1+\epsilon_0$. So, $M$ covers $(1-\epsilon_0)-\epsilon_0(1+\epsilon_0)$ many vertices. This can be made as close to $1$ as we like. 
\end{proof}

And, we can recover the finitary theorem:
\begin{cor}
For every $\epsilon>0$, there is some tolerance $\epsilon$, size $n$, and degree $\Delta$ so that every hypergraph $H$ with $n$ vertices which is $\delta$-good for $\Delta$ admits a matching that covers $(1-\epsilon)n$ vertices.
\end{cor}
\begin{proof}
    Otherwise, for every $\delta$ and $\Delta$ we can find arbitrarily large $\s H$ satisfying $(1-3)$ with no such matching. But then, taking a local-global limit along a subsequence gives an atomless pmp hypergraph with no matching covering a set of vertices of measure $(1-\epsilon)$. But this contradicts the previous theorem.
\end{proof}

\subsection{Frankl--R\"odl again}

In this section we'll give a second proof of the Frankl--R\"odl theorem in the special case of linear hypergraphs with few short cycles. (In fact, we'll get a lower bound on the size of a matching for such hypergraphs in a fixed degree). The point of this argument is to illustrate how the differential equation method adapts to the measurable setting.

Our argument will essentially be the same as above: randomly move edges into a matching with probability $\epsilon/\bb E(\deg)$ and delete the vertices in edges we move.
This time, we'll track the number of edges and vertices remaining by a differential equation.

It is illuminating to compare the argument below with standard differential equation arguments, such as those in Bennett and Dudek's survey, \cite{DiffEqsSurvey}. Our method only works in the bounded degree regime, and our results are comparatively soft. But, the analysis involved in our argument is also comparatively soft. We don't need any concentration estimates or any analytical tools more complicated than Euler approximation. And various geometric considerations are simplified in the measurable setting: our process is independent on different branches of the hypergraph in each step, if we were to edge-color rather than find a matching our process would be symmetric under permuting the colors, etc. (Similar facts are only approximately true in the finite setting).

\begin{dfn}
    For $u,d\in \N$, let $q_{u,d}$ be solutions to the system of differential equations
    \[q(0)=1\quad q'=-\frac{1}d -(u-1-\frac{u}{d})q.\]
    Let $t_{u,d}$ be the solution to $q(t_{u,d})=0$.
\end{dfn}

Note that this equation is separable and has solution
\[q_{u,d}(t)=\frac{ud-d}{ud-d-1} e^{-(u-1-\frac{1}{d})t}-\frac{1}{ud-d-1}.\] So, \[t_{u,d}=\frac{\log(ud-d)}{u-1-\frac{1}{d}}.\] In particular, for fixed $u$, $t_{u,d}$ goes to infinity with $d$.

\begin{thm}
    If $\s H$ is $u$-uniform $d$-regular acyclic pmp hypergraph then, for any $\epsilon>0$, $\s H$ has a matching covering a set of vertices of measure $1-e^{-t_{u,d}}-\epsilon$.
\end{thm}

This measure works out to (the hideous expression) \[1-\left(\frac{1}{(u-1)d}\right)^{u-1-\frac{1}{d}}.\] Note that this tends to $1$ with $d$.
\begin{proof}
By the hypergraph Ab\'ert--Weiss theorem, it suffices to prove this when $\s H$ is the Bernoulli hypergraphing over the (Dirac mass at the) $u$-uniform $d$-regular acyclic hypergraph, $H_{u,d}$. For ease of notation, let's fix $u,d$ and set $H=H_{u,d}$, $q=q_{u,d}$, and $t_0=t_{u,d}$. Fix any vertex $v$ of $H$ and edge $e$ of $H$ which contains $v$. We want to build an automorphism fiid edge labelling of $H$, $\bf f$, so that
\[\bb P\bigl(\bigvee_{v\in f\in H} \bf f(f)=1\bigr)=1-e^{-t_0},\] or close to it anyways. (The choice of $v$ and $e$ doesn't matter since the automorphism group of $H$ acts transitively). 

We'll first define a sequence of fiid labellings, $\ip{\bf f_i:i\in\N}$, by induction. Roughly speaking, in step $i$ we'll label each remaining edge $1$ independently with probability $\epsilon/dq(\epsilon i)$, then label the vertices in these edges $0$, then label the edges containing vertices labelled $0$ with $0$. So, after each step, no unlabelled edge contains a labelled vertex and no unlabelled vertex is contained in an edge labelled $1$. There may be some intersecting edges labelled $1$, but only with negligible density.

We'll show that, for $\epsilon$ small enough, 
\[Q(i):=\bb P(e\mbox{ is unlabelled at step }i\;|\; v\mbox{ is unlabelled at step }i)\]
\[V(i):=\bb P(e\mbox{ is unlabelled at step }i), \mbox{ and }\]
\[C(i):=\bb P\bigl(\bigvee_{v\in f\in H} \bf f_i(e)=1\bigr)\] stay close to $q(\epsilon i), e^{-\epsilon i}$, and $1-e^{-\epsilon i}$ respectively. 

Formally, fix $\epsilon>0$ and define $\ip{\bf f_i:i\in\N}$, induction as follows:
\[\bf f_0=\emptyset\mbox{ with probability }1\]
as above $Q(i)=\bb P\bigl(e\not\in \dom(\bf f_i)| v\not\in \dom(\bf f_i)\bigr).$ If $Q(i)<\epsilon$, we end the construction. Otherwise, Let $\bf A'_i\subseteq H$ be a random set of edges so that, for each $f$, $f\in \bf A'_i$ iid with probability $\epsilon/dQ(i)$ . Let $\bf A_i=\bf A_i'\sm \dom(\bf f_i)$, and define $\bf f_{i+1}$ on vertices by
\[\bf f_{i+1}(v)=\left\{\begin{array}{ll} 0 & (\exists f\in \bf A_i)\; v\in f \\ \bf f_i(v) & \mbox{ else}\end{array}\right.\] And, define $\bf f_{i+1}$ on edges by
\[\bf f_{i+1}(f)= \left\{\begin{array}{ll} 0 & f\not\in \dom(\bf f_i) \;\&\;(\exists f'\in \bf A_i)\;f'\cap f\not=\emptyset\;\&\; f'\not=f \\ \bf f_{i}(f) &  f\in \dom(\bf f_i)\vee f\not\in \bf A_i \\ 1 & \mbox{otherwise}  \end{array}\right.\]

Define $Q(i)$, $V(i)$, and $C(i)$ as above. Note that $Q, V,$ and $C$ are not random variables. Let $\bf\deg_i(v)$ be the degree of $v$ in the unlabelled graph at step $i$. That is,
\[\bf\deg_i(v)=\sum_{v\in f} \bf 1_{f\not\in \dom(\bf f_i)}.\] So, we can compute expected degree in a few ways:
\begin{align*}
    \bb E(\bf \deg_i(v)\;|\;v\not\in\dom(\bf f_i))&=\sum_{m=0}^d m\bb P(\bf \deg_i(v)=m| v\not\in\dom(\bf f_i)) \\ 
    &=\sum_{e\in f} \bb P(\bf 1_{f\not\in \dom(\bf f_i)}| v\not\in\dom(\bf f_i))=d Q(i).
\end{align*}

Now we'll derive difference equations for $Q, V,$ and $C$ and observe that they are essentially the defining equations for Euler approximations to $q, e^{-t}$, and $1-e^{-t}$. Note that, for any finite radius the probability there are two edges in $\bf A_i$ within this radius is $O(\epsilon^2/Q(i))$ (where the hidden constant only depends on the radius). This probability will be negligible for most of the run time of the algorithm. 

First, let's do $V$: 
\begin{align*} V(i+1)& = V(i)-\sum_{v\in f}\bb P(f\in \bf A_i\;|\;v\not\in \dom(\bf f_i))V(i)+O(\epsilon^2)\\
& = V(i)-\sum_{v\in f} \frac{\epsilon }{dQ(i)}\bb P(f\not\in \dom(\bf f_i)\;|\; v\not\in\dom(\bf f_i)) V(i)+O(\epsilon^2)\\
& = V(i)-\frac{dQ(i)}{dQ(i)}\epsilon V(i)+O(\epsilon^2) \\
& = V(i)-\epsilon V(i)+O(\epsilon^2)\end{align*} 

For $C$, let $E$ be the event that there are two edges within radius $3$ of $v$ which are both in $\bf A_i$. Then, for some $E', E''\subseteq E$,
\begin{align*}
    C(i+1)&= C(i)+V(i)\bb P\bigl(\bigvee_{v\in f} f\in \bf A_i\;|\; v\not\in \dom(\bf f_i)\bigr)+\bb P(E') \\
    &= C(i)+V(i)\sum_{v\in f} \bb P\bigl(f\in \bf A_i'\bigr)\bb P\bigl(f\not\in \dom(\bf f_i)\;|\; v\not\in \dom(\bf f_i)\bigr)+\bb P(E'') \\
    & = C(i)+V(i)\sum_{v\in f} \frac{dQ(i)\epsilon}{Q(i)d} +O(\epsilon^2) \\
    & = C(i)+\epsilon V(i) +O(\epsilon^2).
\end{align*}

For $Q$, let's introduce an auxiliary function, $P(i)=\bb P(e\not\in \dom(\bf f_i))$. If $v$ is labelled so is $e$. So we have
\[Q(i)=\frac{\bb P\bigl(e\not\in \dom(\bf f_i)\;\&\; v\not\in \dom(\bf f_i)\bigr)}{V(i)}=\frac{P(i)}{V(i)}.\] Similar to the above, we have
\[P(i+1)=P(i)-\frac{\epsilon}{dQ(i)} P(i)-\frac{\epsilon u (d-1)}{d}P(i)+O(\epsilon^2).\] By the binomial theorem,
\[P(i+1)=P(i)\bigl(1-\frac{\epsilon}{dQ(i)}\bigr)\bigl(1-\frac{\epsilon}{d}\bigr)^{u(d-1)}+O(\epsilon^2)\]
\[V(i+1)=V(i)\bigl(1-\frac{\epsilon}{d}\bigr)^{d} +O(\epsilon^2).\] So, as long as we're in a regime where $V(i)$ stays bounded above $\epsilon$, dividing these equations gives
\begin{align*} Q(i+1)&=\frac{P(i+1)}{V(i+1)} \\
                    & =Q(i)\bigl(1-\frac{\epsilon}{dQ(i)}\bigr)\bigl(1-\frac{1}{d}\epsilon\bigr)^{ud-u-d} +O(\epsilon^2) \\
                    & =Q(i)-\frac{\epsilon}{d}-\frac{ud-u-d}{d}\epsilon Q(i)+O(\epsilon^2).\end{align*}

And, by definition we have $Q(0)=V(0)=1$ and $C(0)=0$. These are are, up to rescaling time and an $O(\epsilon^2)$ error, the equations defining the Euler approximations to 
\[v'=-v\quad c'=v\quad q'=-\frac{1}{d}-(u-1-\frac{u}{d})q\]
\[q(0)=v(0)=1,\quad c(0)=0.\] Let $q,v,c$ solve these equations. These are all Lipshitz functions, so an exercise in real analysis shows that this $O(\epsilon^2)$ error doesn't throw off the convergence of these approximations. See \cite{coloringtrees} for an example of this worked out in detail. We then have that, for any $\delta$ and any $t$ where $q(s)>0$ and $v(s)>0$ for all $s<t$, there is an $\epsilon$ so that for all $i<t/\epsilon$
\[|Q(i)-q(i\epsilon)|, |V(i)-v(i\epsilon)|, |C(i)-c(i\epsilon)|<\delta.\] In this case, we can solve these equations rather simply. We have that $v(t)=e^{-t}$, $q(t)$ is strictly decreasing and bounded above by $v(t)$, and $c(t)=1-e^{-t}$. So, if $t$ is any time strictly less than $t_0$, there is an $\epsilon$ and an $i$ with $|C(i)-c(t)|<\delta$. That is, there are matchings so that the probability a vertex is covered is arbitrarily close to $1-e^{t_0}$, as desired.

\end{proof}

We could improve this theorem in a couple ways. For instance, a more careful analysis of the proof reveals that we just need an edge transitive atomless pmp hypergraph with large enough girth. And, if we worked a little harder in the analysis we could replace transitivity with a concentration argument to control the degrees.

Finally, we can effortlessly recover finitary results about random regular hypergraphs via the usual kind of compactness argument.

\begin{cor}
    If $\s H_n$ is a random $u$-uniform $d$-regular hypergraph with at least $n$ vertices, then asymptotically almost surely, $\s H_i$ has a matching covering a $(1-e^{-t_{u,d}})$-fraction of its vertices.
\end{cor}
\begin{proof}
    Suppose not. Then there's a convergent subsequence with small matching ratio. The sequence is asymptotically acyclic, so the limit has a big matching. This is a contradiction.
\end{proof}

\section{Open questions}

We'll end with a handful of open questions. In \cite{LocalGlobalBackground}, it's shown that when $\Gamma$ has $T_d$ as its Cayley graph, there are no differences in local-statistics between $\aut(T_d)$- and $\Gamma$-fiid processes. It isn't clear to what extent this generalizes to hypergraphs.

\begin{prb}
    Does $H_{d,u}$ have an approximate $\aut(H_{d,u})$-fiid $d$-edge coloring?
\end{prb} That is, if $\mu_0$ is concentrated on $H_{d,u}$, is $\s B(\mu_0)$ local-global equivalent to a Schreier graph of $(\Z/u\Z)^{*d}$?

If this is false (which back-of-the-envelope entropy calculations suggest is the case), then it leaves open the extend to which $\aut(H_{u,d})$- and $\aut(\Gamma)$-fiid processes differ. For instance, the independence ratio of $\s B(\mu_0)$ is known asymptotically in $d$ for each $u$, but, we don't know the independence ratio of the corresponding Bernoulli shifts:

\begin{prb}
    What is the independence ratio of the Schreier graph of the Bernoulli shift of $(\Z/u\Z)^{*d}$?
\end{prb}

A major open problem for graphs is to characterize the local-global limits of finite graphs. We can ask questions in the same vein for hypergraphs. For instance:
\begin{prb}
    Is there a regular acyclic pmp hypergraph which is not a local-global limit of finite hypergraphs.
\end{prb}

Lastly, this paper only leveraged the convergence of $\Sigma_1$ properties of hypergraphs, but we have convergence of the entire first order theory. It remains open if there are any higher-complexity properties of hypergraphs in the literature which can be proven using the limit machinery.

 \bibliographystyle{plain}
\bibliography{refs}

\end{document}